\documentclass{article}
\usepackage[utf8]{inputenc}
\usepackage{fullpage}
\usepackage{mathtools}
\usepackage{dsfont}
\mathtoolsset{showonlyrefs,showmanualtags}

\usepackage{amssymb}
\usepackage{amsthm}
\usepackage{comment}
\usepackage{stmaryrd}
\usepackage{enumerate}
\usepackage{tikz}
\usepackage{mathtools}
\usepackage{mathrsfs}
\usepackage[scr=rsfs,cal=boondox]{mathalfa}
\usepackage[hidelinks]{hyperref}

\newcommand\simeqreim{\stackrel{{\text{Re\slash{}Im}}}{\simeq}}
\newcommand\simeqcasida{\stackrel{{\text{Casida}}}{\simeq}}

\usepackage[activeacute,english]{babel}

\newcommand\R{\mathbb R}
\newcommand\C{\mathbb C}
\newcommand\ii{{\rm i}}
\newcommand\jj{{\rm j}}
\newcommand\Uperp{X_{\perp}}

\newcommand\realvs{Y_{\mathbb{R}}} 
\newcommand\complexvs{Y} 
\newtheorem{lemma}{Lemma}
\newtheorem{theorem}{Theorem}
\newtheorem{proposition}{Proposition}

\newtheorem{remark}{Remark}
\newtheorem{assumption}{Assumption}

\title{Linear response and resonances in adiabatic time-dependent
  density functional theory}
\author{Mi-Song Dupuy$^1$, Eloïse Letournel$^2$, Antoine Levitt$^3$}
\begin{document}
\maketitle

\begin{abstract}
  We consider the electrons of a molecule in the adiabatic
  time-dependent density functional theory approximation. We establish
  the well-posedness of the time evolution and its linear response
  close to a non-degenerate ground state, and prove the appearance of
  resonances at relevant frequencies. The main mathematical difficulty
  is due to the structure of the linearized equations, which are not
  complex-linear. We bypass this difficulty by reformulating the
  linearized problem as a real Hamiltonian system, whose stability is
  ensured by the second-order optimality conditions on the energy.
  \end{abstract}
\footnotetext[1]{Laboratoire Jacques-Louis Lions, Sorbonne Université, Paris, France}
\footnotetext[2]{Inria Paris and Universit\'e Paris-Est, CERMICS, Ecole des Ponts ParisTech, Marne-la-Vall\'ee, France}
\footnotetext[3]{LMO, Université Paris-Saclay, Orsay, France}
\tableofcontents


\section{Introduction}

The Density Functional Theory (DFT) in the Kohn-Sham formalism
\cite{hohenberg_kohn,kohn_sham} decribes the ground-state of a static
interacting system through a fictitious model of $N$ non-interacting
electrons, coupled by a mean field. For $N$ electrons (which we take
spinless for simplicity of notation), the equations for the orbitals
$\Psi^0=(\psi_i^0)_{i=1}^N$ are:
\begin{align}
  \label{eq:stationary_state}
  H[\rho_{\Psi^0}] \psi_i^0&= \lambda_i \psi_i^0,
\end{align}
where
\begin{align}
    H[\rho]&=\underbrace{-\frac{1}{2}\Delta}_\text{kinetic operator}+\underbrace{V_{\rm ext}}_\text{external (atomic) potential} +\underbrace{\left(\rho*\tfrac{1}{|r|}\right)}_{\text{Hartree potential $V_H$}} + \underbrace{v_{\rm xc}[\rho]}_{\text{exchange-correlation potential}}\\
    \rho_{\Psi^0}&=\sum_{i=1}^N |\psi_i^0|^2
\end{align}
Formally exact, DFT models are in practice approximated by an explicit
choice of the exchange-correlation potential $v_{\rm xc}$ (for instance,
the local density approximation \cite{lewin_lda}).

Consider a system in its ground state, described by the orbitals
$\Psi^{0}$. We wish to study how the system is perturbed by the
addition of a time-dependent potential $f(t) V_{P}$. The Runge-Gross
theorem \cite{runge_gross} forms the theoretical basis of time-dependent DFT (TDDFT) by ensuring that the electronic density of the time-dependent
interacting system can be described through the time-dependent
Schrödinger equation of a non-interacting system, again coupled by a
mean-field. Similarly to the static case, there exists an exact
exchange-correlation potential $v_{\rm xc}[\rho]$, but this time the
potential is non-local in time (depends on the values of the density
at all previous times), making accurate approximations much more
difficult. In this work, we assume the adiabatic approximation, in
which $v_{\rm xc}$ only depends on $\rho(t)$. Each orbital $\psi_i$
then satisfies the equation
\begin{align} \label{eq:TDDFT}
  \boxed{\begin{cases} \jj \partial_t \psi_i(t)&= (H[\rho_{\Psi(t)}]+ \varepsilon f(t) V_P)\psi_i(t)\\
        \psi_i(0)&= \psi_i^0 \end{cases}}
\end{align}
which will be the main object of our study. For reasons that will be
made clear shortly, we use the notation $\jj$ instead of the usual
$\ii$ for the imaginary unit in the Schrödinger equation; we reserve
the notation $\ii$ for a subsequent imaginary unit, to be used in time
Fourier transforms and resolvents.

\medskip

The TDDFT, even within the adiabatic simplification, contains an
enormous amount of physics, and is sufficient to model the
interaction of light with electrons, usually at a good qualitative and
sometimes quantitative level \cite{marques_tddft}. To extract the
relevant features (excitation energies, absorption spectra...), it is
very useful to linearize the equations around a solution of the
stationary equations
\eqref{eq:stationary_state}, which yield a dynamical solution
$\psi_{i}^{0} e^{-\jj \lambda_{i} t}$. Setting
\begin{align}
  \psi_i(t)&= e^{-\jj \lambda_i t} \left(\psi_i^0+ \varepsilon u_i^1(t)  + \dots \right)\label{eq:sternheimer_intro},
\end{align}
and truncating to first order, we obtain formally the linear response
TDDFT (LR-TDDFT) equations
\begin{align}\label{eq:casida_step_1}
  \jj\partial_t  u_i^1(t) &= \underbrace{\big(H[\rho_0] -\lambda_i\big) u_i^1(t)+  \left( \frac{dH}{d\rho} \frac{d\rho}{d\Psi} U^{1}(t) \right) \psi_i^0}_{(\mathcal M_{\rm dyn}(U^{1}))_{i}}  + f(t) V_P \psi_i^0 
\end{align}
where the derivatives are evaluated at $\Psi^{0}$, and where
$U^{1}(t) = (u_{i}^{1}(t))_{i=1}^{N}$. From this a number of useful
properties can be obtained, such as the susceptibility operator
(density-density response function) $\chi(t)$, a real-valued map
(containing no $\jj$) from
potential variations to density variations defined by $\chi(t) = 0$ on
$\R^{-}$ and
\begin{align}
  \frac{d\rho}{d\Psi} U^{1}(t) = \int_{0}^{t} \chi(t-t') V_{P} f(t') dt'.
\end{align}
Of particular interest is its Fourier transform
\begin{align}
  \label{eq:convention_fourier_transform}
  \widehat \chi(\omega) = \int_{-\infty}^{\infty} \chi(t) e^{\ii\omega t} dt,
\end{align}
interpreted as the response of the system to a periodic excitation
$e^{-\ii\omega t}$ (note the unusual sign convention, classical in
quantum mechanics, and the use of the imaginary unit $\ii$, distinct
from $\jj$). Because the distributional Fourier transform of $-i \theta(t) e^{-\ii \lambda t}$ is given
by
$\lim_{\eta \to 0^{+}} \frac 1 {\omega+\ii \eta - \lambda} = {\rm p.v.} \frac 1 {\omega-\lambda} - \ii \pi \delta(\omega-\lambda)$,
the eigenvalues of the operator $-\jj \mathcal M_{\rm dyn}$ appearing
in the linearized equation become singularities of $\widehat \chi(\omega)$; in
particular, the imaginary part of $\widehat \chi(\omega)$ has Dirac
peaks corresponding to each eigenvalue. An exception are ``hidden''
eigenvalues of $M_{\rm dyn}$ corresponding to modes that are never
excited or observed, including in particular transitions
between occupied eigenvectors.

A difficulty in realizing this program in practice is that the
operator $\mathcal M_{\rm dyn}$ is not a complex-linear operator (it does not
commute with $\jj$). This is because of the form of the density mapping
\begin{align}
  \left(\frac{d\rho}{d\Psi} U\right)(r) = \sum_{j=1}^N \overline{\psi_j}(r) u_j(r) + {\psi_j}(r) \overline{u_j}(r)
\end{align}
which involves a $\overline{u_{i}}$ term. This means that one cannot
solve the linearized equation by an exponential in the complex vector
space $(L^{2}(\R^{3}, \C))^{N}$, as is done in the
$\C$-linear case. This problem is usually bypassed by writing a
coupled equation for $U$ and $\overline U$ (involving a complex-linear
operator), which when discretized in a finite basis becomes the Casida
equations \cite{casida}. Mathematically, we will handle this by
working in the complexification
$(L^{2}(\R^{3},\C)^{N},\R) + \ii (L^{2}(\R^{3},\C)^{N},\R)$, where
$(L^{2}(\R^{3},\C)^{N},\R)$ is seen as a vector space over the real numbers
(and not the complex numbers, as usually done). This explains our use
of separate imaginary units for the unit $\jj$ of $L^{2}(\R^{3},\C)^{N}$
and $\ii$ of the complexification.

When there is no electron-electron interaction,
$(\mathcal M_{\rm dyn} U)_{i} = H[\rho_{0}]-\lambda_{i}$, and the
eigenfrequencies are simply the one-electron excitation energies,
given by the differences between the unoccupied part of the spectrum
$\sigma(H[\rho_{0}])$ and the occupied energies $\lambda_{i}$. Because
of the continuous spectrum $[0,\infty)$ of the effective Hamiltonian
$H[\rho_{0}]$, which is involved in the spectrum of
$\mathcal M_{\rm dyn}$ at energies greater than the ionization
threshold
$\lambda_{\rm ion} = -\max(\lambda_{1}, \dots, \lambda_{N})$, the
distribution ${\rm Im}(\widehat \chi(\omega))$ includes a continuous
part. This is related to the physical process of photoionization,
where light shined on a molecule ionizes an electron. It sometimes
happens that both the sharp Dirac peaks (related to bound state to
bound state excitations) and the continuous part (related to bound
state to scattering state transitions, i.e. ionization) are present at
the same frequencies, as in the case of the 1s$\to$ 2p excitation and
the 2s ionization in Beryllium \cite{toulouse}. When electron-electron interaction
is present, both features merge into a single resonance, a mathematically subtle phenomenon.

\medskip
In this paper,
\begin{itemize}
\item we formalize the mathematical framework required to treat the
  (real-linear) linearized equations, generalizing the Casida formalism;
\item we prove the well-posedness of \eqref{eq:TDDFT} for small
  $\varepsilon$ and finite times, allowing us to define $\chi(t)$
  rigorously;
\item we study the complex structure of $\widehat \chi$, showing in
  particular that the ionization threshold (first branch cut of
  $\chi$) is the same as that of the non-interacting system;
\item we show that, when the electronic interaction strength is small
  and when there are excitations and ionization processes at the same frequency,
  resonances appear in the analytic continuation of $\widehat \chi$.
\end{itemize}
We establish these results for a molecular system with a LDA-type
exchange-correlation potential, but the results of this paper would
apply (under suitable assumptions) to hybrid functionals (Hamiltonians
depending on $\Psi$, not just $\rho_{\Psi}$), or systems subject to
magnetic fields.

Our well-posedness analysis is based on the identification of (a part
of) the operator $\mathcal M_{\rm dyn}$ with the Hessian of the energy
of the ground-state problem. Under the crucial assumption that the
ground state is a non-degenerate local minimum of the energy, the
linearized equation \eqref{eq:casida_step_1} then appears as the
linearization of a Hamiltonian system at a stable equilibrium, which
possesses some degree of stability with respect to external
perturbations. Compared to the existing literature, we establish
well-posedness without assuming smallness of $v_{\rm xc}$, or
requiring an ad-hoc stability condition. Our study of resonances is
based on a Fermi golden rule applied to an operator related to
$\mathcal M_{\rm dyn}$. As far as we are aware, this is the first
rigorous result on resonances in a mean-field context.



\medskip

The first step in the mathematical study of our problem is to
establish the existence of a static solution $\Psi^{0}$. This is
usually obtained by variational methods; see
\cite{anantharaman2009existence} for the case we consider here. Then,
the well-posedness of various forms of the nonlinear Schrödinger
evolution equation is usually established using fixed-point arguments
combined with energy estimates; see for instance \cite{cances-lebris,
  dietze2021dispersive} for equations related to ours, and
\cite{sigal} for the case of the Kohn-Sham equations. Our objective in
this work is to study the linear response in the regime where
$\varepsilon$ is small, $t$ is large, but without imposing a
restriction on the size of $v_{\rm xc}$, which is not covered by
existing results.

The linear response regime has been studied in various works,
including \cite{Weinstein} for the nonlinear Schrödinger equation
(without external perturbation), and \cite{cances-stoltz} for a defect
in a periodic system in the Hartree model. In both cases, the
well-posedness of the linearized equation was established using
information on the sign of the nonlinearity, which is not possible for
general Kohn-Sham equations. The importance of second-order
information to establish various properties of the model was recently
emphasized in \cite{kemlin}.

The particular question of the linear response under the action of a
time-dependent source has also been studied. The first and third
authors of this work have studied the effects of a finite-size domain
in the simulation in \cite{dupuy2021finitesize}. In
\cite{corso_density-density_2023} and \cite{corso}, the
density-density response functions are studied via the perturbative
Dyson framework. The structure of the linear response equations have
also been studied mathematically in a linear algebraic context in \cite{bai2012minimization}.

The mathematical theory of resonances, pioneered by \cite{aguilar_combes}, is by now a well-established subject, summarized recently in \cite{dyatlov_mathematical_2019}. The particular type of resonances of interest here, arising from the perturbative interaction between bound and scattering states, has been studied in various physical settings and mathematical formalisms, notably in the time domain \cite{davies1974markovian, soffer1998time} and in the frequency domain \cite{barry}. Numerically,
resonances have been investigated in numerous works, with the closest
to our topic being \cite{toulouse} on the example of the resonance of
the Beryllium atom.

%
%

 \medskip

 We study the properties of the ground state and explain our
 complexification of the problem in in Section \ref{sec_system}. We then state our main results in Section \ref{sec_main_results}. The proofs of the well-posedness and behaviour at first order are gathered Section \ref{sec_well-posedness}, and Section \ref{sec:propagator_and_response} focuses on the study of the propagator and the frequency response function with its resonances.

\section{System of study} \label{sec_system}
\subsection{Notation}
We study a system of $N$ spinless electrons in a mean-field model. The
one-electron Hilbert space is the $\mathbb{C}$-vector space
$L^2(\mathbb{R}^3, \mathbb{C})$, endowed with the usual scalar product
\begin{align}
    \forall u, v \in L^2(\mathbb{R}^3, \mathbb{C}), \langle u |v \rangle = \int_{\mathbb{R}^3} \overline{u (r)} v (r)dr
\end{align}
We will often simply denote $L^2(\mathbb{R}^3, \mathbb{C}) = L^2$ and the
Sobolev space $H^2(\mathbb{R}^3, \mathbb{C}) = H^2$. The sets of
orbitals $\Psi = (\psi_{i})_{i=1}^{N}$ belong to the space
$(L^{2})^{N}$ of $N$-tuples of functions, with inner product
\begin{align}
  \langle\Psi|U\rangle&= \sum_{i=1}^N \langle \psi_i|u_i\rangle
\end{align}
We define the manifold of the orbitals $\mathfrak{M}_N$:
\begin{align}
  \mathfrak{M}_N&=\Big\{\Psi \in (L^2)^N:~\langle \psi_i|\psi_j\rangle=\delta_{ij},~\forall (i,j)\in \{1,... ,N\}^2\Big\}.
\end{align}
Any operator $A$ acting on $L^2$, such as a one-particle Hamiltonian, extends naturally to an operator acting on
$(L^{2})^{N}$ through
\begin{align}
  \label{eq:AU}
  (AU)_i=A(U_i), \quad \text{ for all } i =1, \dots, N.
\end{align}
Similarly, if $R$ is a $N \times N$ matrix,
we define
\begin{align}
  (U R)_{i} = \sum_{j=1}^{N} U_{j} R_{ji}.
\end{align}


We will quantify exponential decay using the spaces
\begin{align}
  L^2_{\alpha}&:= \{ u \in L^2: ~ r\mapsto u(r) e^{\alpha \langle r\rangle}\in L^2\} \\
  H^2_{\alpha}&:= \{ u \in L^2: ~ r\mapsto u(r) e^{\alpha  \langle r \rangle} \in H^2\} 
\end{align}
for $\alpha \in \mathbb{R}$, where we use the japanese bracket
notation $\langle r \rangle = \sqrt{|r|^2 + 1}$ for the regularized
absolute value.

In proofs, the notation $C$ will refer to an unimportant constant that may change from line to line.
\subsection{The static problem}
The static problem is to find the minimum of $\mathcal{E}(\Psi)$ for $\Psi$ in $\mathfrak{M}_N$, with $\mathcal{E}$ the energy
\begin{align}
    \mathcal{E}(\Psi)&=\sum_{i=1}^N\left( \frac{1}{2}\int_{\mathbb{R}^3} |\nabla \psi_i(r)|^2 + \int_{\mathbb{R}^3} V_{\rm ext}(r)|\psi_i(r)|^2 dr \right)+ \frac{1}{2} \int_{\mathbb{R}^3} \frac{\rho_{\Psi}(r) \rho_{\Psi}(r')}{|r-r'|}drdr' + \mathcal{E}_{\rm xc}[\rho_{\Psi}]\end{align}
    where $\mathcal{E}_{\rm xc}[\rho]= \int_{\mathbb{R}^3}e_{\rm xc}(\rho(r))dr$ and $e_{\rm xc}:\mathbb{R}^+ \rightarrow \mathbb{R}$.

\begin{assumption}[Conditions on $V_{\rm ext}$]\label{assumption1} The
  potential $V_{\rm ext}$ is $L^2+ L_{\varepsilon}^\infty$: for all
  $\varepsilon > 0$, there is a decomposition
  $V = V_{2,\varepsilon}+V_{\infty,\varepsilon}$ with $V_{2,\varepsilon} \in L^{2}(\R^{3},\R)$
  and $V_{\infty,\varepsilon} \in L^{\infty}(\R^{3},\R)$ with
  $\|V_{\infty,\varepsilon}\|_{\infty} \le \varepsilon$.
\end{assumption}
This is a rather standard assumption, which allows in particular
Coulomb potentials. It ensures that $V_{\rm ext}$ is a
$\Delta$-compact operator, and therefore that
$-\tfrac 1 2 \Delta + V_{\rm ext}$ is self-adjoint on $L^{2}$ with
domain $H^{2}$, and essential spectrum $[0,\infty)$ \cite[Theorem X.15]{ReedSimonII}.

\begin{assumption}[LDA approximation]\label{assumption2}
  The exchange-correlation energy $e_{\rm xc}:\mathbb{R}^+ \rightarrow \mathbb{R}$ is $\mathcal{C}^4$ and $e_{\rm xc}(0)=e_{\rm xc}'(0)=0$.
\end{assumption}
This in particular does not allow to use the homogeneous electron
gas LDA exchange approximation, as it is only $\mathcal{C}^1$ at $0$. The
arguments in this paper could possibly be adapted to treat this case,
using a careful analysis of density tails. However, the LDA
approximation is not expected to be accurate on regions where the
density is small (where correlation is strong), and therefore to avoid
unnecessary complications we will not do so.

Let $H[\rho]$ be the mean-field Hamiltonian:
\begin{gather}\label{eq:Hamiltonian}
    H[\rho]= -\frac{1}{2} \Delta+V_{\rm ext} +v_{\rm hxc}(\rho),
\end{gather}
where
\begin{align}
  v_{\rm hxc}(\rho) &= \left(\rho*\tfrac{1}{|r|}\right) + v_{\rm xc}(\rho), \ \text{and } 
  v_{\rm xc}(\rho) = e_{\rm xc}'(\rho).
\end{align}
If $\Psi$ is in $(H^{2})^{N}$, since $H^{2}$ is an algebra, $\rho$ is
in $H^{2}$, hence $v_{\rm hxc}(\rho) \in L^\infty$ and $H[\rho]$ is self-adjoint on $L^{2}$ with
domain $H^{2}$.

For $\Psi\in \mathfrak{M}_N$, we define the $\mathbb{R}$-linear operators $\mathcal{S}_{\Psi} : (H^2)^N \to H^2$ and $\mathcal{K}_{\Psi} : (H^2)^N \to (H^2)^N$ as follows, for all $U \in (H^2)^N$
\begin{align}\label{eq:Spsi}
  \mathcal{S}_{\Psi}(U)&= \frac{d\rho{}}{d\Psi{}} U = \sum_{j=1}^N \overline{\psi_j} u_j + {\psi_j} \overline{u_j},\\
  (\mathcal{K}_{\Psi}(U))_i&= \left(\frac{dv_{\rm hxc}}{d\rho}\mathcal{S}_{\Psi} (U)\right) \psi_{i}= \left(\left(\mathcal S_{\Psi}(U)*\frac{1}{|r|}\right)+  v_{\rm xc}'\left(\rho_{\Psi}\right) \mathcal{S}_{\Psi}(U)\right) \psi_i . \label{eq:def_K}\\
\end{align}

We can then compute
\begin{proposition}[Second-order expansion of the energy]\label{generalites_E_H_K}
    ~

    The energy $\mathcal{E}$ is $\mathcal C^{2}$ on $(H^{2})^{N}$.
    For any $\Psi, U$ in $(H^2)^N$,
        \begin{align}\boxed{
            \mathcal{E}(\Psi+U)= \mathcal{E}(\Psi)+ 2{\rm Re}\langle H[\rho_\Psi]\Psi |U \rangle +  \Big(\langle U |H[\rho_{\Psi}] U\rangle+ {\rm Re}\langle U | \mathcal{K}_{\Psi}(U) \rangle\Big)+o(\|U\|_{H^2}^2)}.
        \end{align}
    \end{proposition}
This expansion is a direct computation proved in Appendix \ref{sec:proof_generalites_E_H_K}.
\\

We now wish to define a notion of non-degenerate local minimum of
$\mathcal E$ on $\mathfrak{M}_N$. Clearly, no local minimum $\Psi$ can
be non-degenerate in the usual sense, since, for any unitary matrix $R \in U(N)$,
$E(\Psi)= E(\Psi R)$. Rather, we assume
\begin{assumption}[Non-degeneracy of the minimum of the energy.]\label{assumption3}
$\Psi^0$ is a local non-degenerate minimum of the energy on
$(H^2)^N\cap\mathfrak{M}_N$, in the sense that there exists $\gamma >0$
and a neighborhood $W$ of $\Psi^{0}$ in $\mathfrak{M}_{N}$ for the
$(H^{2})^{N}$ topology inside which
\begin{align}
  \mathcal{E}(\Psi)-\mathcal{E}(\Psi^0) \geq \gamma \min_{R \in {\rm U}(N)} \|\Psi - \Psi^{0} R\|^{2} \label{eq:non_degeneracy}.
\end{align}

\end{assumption}
The right-hand side measures the distance between the subspaces
spanned by $\Psi$ and $\Psi^{0}$; it is one of several equivalent
measures, another being for instance the norm of the difference of
projectors (see for instance \cite{ye2016schubert} for a review).

This assumption implies in particular that
$\langle H[\rho_{\Psi^{0}}] \Psi^{0}|U \rangle=0$ for all $U$ in the
tangent space
\begin{align}\label{eq:def_tangent_space}
  T_{\Psi^{0}} \mathfrak{M}_N = \left\{U \in (L^{2})^{N}, {\rm Re}\langle  u_{i}|\psi_{j}^{0} \rangle = 0\;\; \forall i,j \in \{1, \dots, N\}\right\},
\end{align}
and therefore that there is a $N \times N$ real symmetric matrix
$(\lambda_{ij})$ such that
\begin{align}
  H[\rho_{\Psi^{0}}] \psi_{j}^{0} = \sum_{i=1}^{N} \psi_{i}^{0} \lambda_{ij}.
\end{align}
After possibly a rotation of the $\psi_{i}$ to diagonalize the $(\lambda_{ij})$ matrix, this can be rewritten as
\begin{align}
  \boxed{H[\rho_{\Psi^{0}}] \psi_{i}^{0} = \lambda_{i} \psi_{i}^{0}}
\end{align}
which we assume in the sequel. Without loss of generality, we also
assume that $\lambda_{1} \le \lambda_{2} \dots \le \lambda_{N}$, but we
do \textit{not} assume that these are the lowest eigenvalues of
$H[\rho_{\psi^{0}}]$ (Aufbau principle). 

\begin{assumption}~
  \label{assumption_negative_eigvals}
  The occupied eigenvalues $\{\lambda_i\}_{i=1}^N$ are negative.
\end{assumption}
This assumption guarantees in particular that the corresponding eigenfunctions $\psi_i^0$ belong to $H^2_\alpha$ for some $\alpha>0$ (see for instance \cite[Lemma 5.1]{dupuy2021finitesize}).

\bigskip

From now on, to simplify the notation, we write
\begin{align}
  \rho_{0} = \rho_{\Psi^0}, \quad H_{0} = H[\rho_0], \quad \mathcal{K}_{0} = \mathcal{K}_{\Psi^0}.
\end{align}
We also write for $P_0$ the orthogonal projector on
${\rm Span}\left(\{\psi_i^{0}\}_{i=1}^N\right)$:
\begin{align}
  \label{eq:16}
  P_{0} = \sum_{i=1}^{N} |\psi_{i}^{0}\rangle \langle  \psi_{i}^{0}|
\end{align}
which acts on $L^{2}$, and therefore also on $(L^{2})^{N}$ by
\eqref{eq:AU}. In particular, we will often use ${\rm Ran}\left(1-P_0\right)$ to
refer to the space of orbital variations $U = (u_{i})_{i=1}^{N}$ which are all
complex-orthogonal to $\Psi^{0}$: $\langle  u_{i}|\psi_{j}^{0} \rangle=0$ for all
$i,j \in \{1, \dots, N\}$. This is the subspace on which Proposition \ref{proposition:operator_M_dyn} gives information:
  \begin{proposition}[Operator $\mathcal{M}_{\rm dyn}$]~ \label{proposition:operator_M_dyn}
      Let $\mathcal{M}_{\rm dyn}$ the $\mathbb{R}$-linear operator acting on
      $ U \in (L^2_{\mathbb{R}})^N$ as
      \begin{align}
        \mathcal{M}_{\rm dyn}&= \Omega+\mathcal{K}_{0} \label{eq:def_M_dyn},
      \end{align}
      with
      \begin{align}\label{eq:def_Omega}
        \Omega U &= (H_{0} - \Lambda) U,\\
        (\Lambda U)_{i} &= \lambda_{i} U_{i}, \quad \text{for }i = 1, \dots, N. \label{eq:def_Lambda}
      \end{align}

      Under Assumption~\ref{assumption3}, for all $U$ in
      $(H^2)^N \cap{\rm Ran}\left(1-P_0\right)$,
          \begin{align}
            \label{eq:truc}
            {\rm Re}\langle U |\mathcal{M}_{\rm dyn} (U)\rangle  \geq  \gamma\|U\|_{L^2}^2.
          \end{align}
  \end{proposition}
  Recall from \eqref{eq:casida_step_1} that $\mathcal{M}_{\rm dyn}$ is also
  the operator which determines the evolution of the dynamic system
  linearized at first order.
  
  The proof of Proposition \ref{proposition:operator_M_dyn} is in
  Appendix \ref{sec:energy}; it is based on a quadratic model
  of \eqref{eq:non_degeneracy} near $\Psi^{0}$, with the quadratic form on
  the left-hand side
  being identified to $\mathcal{M}_{\rm dyn}$, and that on the
  right-hand side to $(1-P_{0})$. The condition \eqref{eq:truc} generalizes to the
  complex case the computations in \cite{kemlin}, from which our
  notation is taken.

\subsection{Time-dependent problem}
The local minimum $\Psi^{0}$
in Assumption \ref{assumption3} induces a stationary solution
\begin{align}
  \psi_i^0(t)=e^{-\jj \lambda_i t}\psi_i^0,
\end{align}
in $(H^2)^N \cap \mathfrak{M}_N$ of the unperturbed Schrödinger equation
\begin{align}
    \jj\partial_t \Psi^0(t) &= H[\rho_{\Psi^0(t)}]\Psi^0(t) = H_{0}\Psi^0(t), \label{eq:TI_schrodinger}
  \end{align}
We now perturb the evolution by a time-dependent multiplicative potential $f(t) V_{P}$. 

\begin{assumption}[Assumption on the perturbative potential]\label{assumption5}
  $f$ is causal ($f(t)=0$ for all negative $t$) and continuous. $V_P$
  is in $H^{2}$.
\end{assumption}
The causality assumption is done to simplify convolutions, since we
are only interested in the behavior for positive times. The regularity of $V_{P}$ is made to ensure a simple well-posedness theory of the
evolution equation in the algebra $H^{2}$; the linear theory (the
definition of the linear response operator $\chi$) requires much less
stringent hypotheses (and in particular accomodates polynomially
growing potentials).

The perturbed Schrödinger equation \eqref{eq:TDDFT} is then
\begin{align}
  {\begin{cases} \jj\partial_t \Psi(t)&= H[\rho_{\Psi(t)}]\Psi(t) + \varepsilon f(t) V_P \Psi(t)\label{eq:TD_schrodinger}\\
        \Psi(0)&= \Psi^{0}. \end{cases}}
\end{align}
To study this equation near the stationary solution
$\psi_i^0(t)=e^{-\jj\lambda_i t}\psi_i^0$, we set
\begin{align}\label{eq:ansatz_u}
  \psi_i(t)&= e^{-\jj\lambda_i t}(\psi_i^0 +\varepsilon u_i(t)).
\end{align}
We then have
\begin{align}  \label{eq:time_evolution_whole_system}\boxed{
    \begin{cases}
      \jj\partial_t U(t) &= \mathcal{M}_{\rm dyn}(U(t))+ f(t) V_P\Psi^0 +\mathcal{R}(U(t), \varepsilon, t) \\
      U(0)&=0
    \end{cases}}
\end{align}
where
\begin{align}
  \mathcal{R}(U, \varepsilon, t)_i&= \varepsilon f(t) V_P  U+ \left(H[\rho_{\Psi^{0}+\varepsilon U}]- H[\rho_0]\right) (\varepsilon^{-1}\Psi^{0}+U)- \mathcal{K}_0(U)
\end{align}
collects all higher order terms.



\subsection{Structure of the space}\label{sec:space_splitting}

\paragraph{Abstract structure.}
In the study of \eqref{eq:time_evolution_whole_system}, we need to
deal with a perturbation of the linear evolution equation
\begin{align}
  \jj \partial_{t} U = \mathcal{M}_{\rm dyn}(U)
\end{align}
We search for a solution of this equation in the set of
time-continuous functions valued in a Hilbert space
\begin{align}
  \label{eq:17}
  \mathcal{Y}:=(L^{2}(\R^{3},\C))^{N}.
\end{align}
To use the formalism of linear algebra and spectral calculus, we must
give this set a vector space structure $(\mathcal{Y}, \mathbb{K})$,
where $\mathbb{K}$ is a scalar field. Choosing either
$\mathbb{K}=\mathbb{C}$ or $\mathbb{K}=\mathbb{R}$ does not change the
solution, but only the tools allowed for the resolution. If we decide
to work in $(\mathcal{Y}, \mathbb{C})$, then $\mathcal{M}_{\rm dyn}$
is not linear, because it does not commute with $\jj$. We thus should
work in the real Hilbert space
\begin{align}
  \realvs= \left(\mathcal{Y}, \mathbb{R}\right)
\end{align}
equipped with the natural inner product
\begin{align}
  \langle U|U' \rangle_{\realvs} = {\rm Re}\langle U|U'\rangle_{\mathcal{Y}}.
\end{align}
$\mathcal{M}_{\rm dyn}$ becomes a linear operator on this space, which we call $M_{\rm dyn}$.
Multiplication by the scalar $\jj$ can be represented as a (real)
linear operator on this space, which we name $J$. We then need to
solve the linear equation
\begin{align}
    J\partial_t U = M_{\rm dyn}U
\end{align}
in $\realvs$. 
However, this structure does not enable us to use the tools of
spectral theory, which requires a complex Hilbert space. In
particular, we would like to exponentiate the unbounded
operator $JM_{\rm dyn} t$, which requires complex functional calculus.
We also study the linear response in frequency domain, so we need to
have a Fourier transform and thus imaginary numbers. To that end, we
will complexify the real Hilbert space $\realvs$, defining the
abstract complex
Hilbert space:
\begin{align}
  \complexvs &= \realvs + \ii \realvs\\
  &= \{U+\ii V: ~~ U, V \in (\mathcal{Y},\R)\}
\end{align}
This naturally has the structure of a complex Hilbert space with
scalar product defined by
\begin{align}
  \langle U+\ii V|U'+\ii V'\rangle_{\complexvs} = \Big(\langle U|U' \rangle_{\realvs} - \langle V|V' \rangle_{\realvs}\Big) + \ii \Big(\langle U|V' \rangle_{\realvs} + \langle V|U' \rangle_{\realvs}\Big).
\end{align}
$(\mathcal{Y},\mathbb{C})$ and $\complexvs$ are both complex Hilbert spaces, but are not isomorphic: for instance, if $(\mathcal{Y},\mathbb{C})$ was finite-dimensional with (complex) dimension $p$, $\realvs$ would have (real) dimension $2p$, and $\complexvs$ would have (complex) dimension $2p$. Note also that the imaginary unit $\jj$ on the complex vector space $((\mathcal{Y},\mathbb{C})$ is distinct from the imaginary unit $\ii$ on $\complexvs$.

\medskip

Any linear operator $L$ on $(\mathcal{Y},\mathbb{R})$ (such as
$M_{\rm dyn}$ or $J$) extends naturally to a complex linear operator on $\complexvs$
(again denoted by the same letter) by setting 
\begin{align}
  \label{eq:18}
  L(U+\ii V)=LU + \ii LV.
\end{align}
For instance, $J$ is an anti-self-adjoint operator on $\realvs$, with no
eigenvalues, and such that $J^2=-{\rm Id}$, and extends to an
anti-self-adjoint operator on $\complexvs$ with eigenvalues $\pm \ii$.


\paragraph{The Casida representation}
In order to actually perform computations on these objects, it is
useful to select a particular representation of the elements of
$\realvs$ and $\complexvs$. The simplest choice is to represent an element by its real and
imaginary parts,
\begin{align}
  U_{\rm r}+\jj U_{\rm j} \simeqreim
  \begin{pmatrix}
    U_{\rm r}\\U_{\rm j}
  \end{pmatrix}.
\end{align}
This effectively establishes a (real) isomorphism between
$\big(\mathcal{Y}, \mathbb{R}\big)= \big(L^2(\mathbb{R}^3, \mathbb{C})^N, \mathbb{R}\big)$ and
$\realvs= \big(L^2(\mathbb{R}^3, \mathbb{R})^N\big)^2$, whose
complexification is then isomorphic to
$\Big(\big(L^2(\mathbb{R}^3, \mathbb{C})^N\big)^2, \mathbb{C}\Big)$.
With this choice, $J$ is represented by the block operator
\begin{align}
  J \simeqreim \begin{pmatrix} 0 & -1 \\ 1&0\end{pmatrix}.
\end{align}
This representation is perfectly adequate for theoretical purposes,
but expressing operators in it can be cumbersome, especially when the
underlying orbitals are complex.

Alternatively, $J$ can be diagonalized in $\complexvs$, yielding an different representation
\begin{align}
  \label{eq:casidavector}
  U_{\rm r}+\jj U_{\rm j} \simeqcasida \frac{1}{\sqrt{2}}\begin{pmatrix} U_{\rm r}+\ii U_{\rm j} \\ U_{\rm r}-\ii U_{\rm j} \end{pmatrix},
\end{align}
in which now
\begin{align}
  J\simeqcasida\begin{pmatrix} \ii & 0 \\ 0 & -\ii \end{pmatrix}
\end{align}
In other words, $\realvs$ is real isomorphic to the real vector space:
\begin{align}
  \realvs\simeqcasida\left( \Big\{\frac 1 {\sqrt 2}\begin{pmatrix} U_{\rm r}+\ii U_{\rm j} \\ U_{\rm r}-\ii U_{\rm j} \end{pmatrix}, (U_{\rm r},U_{\rm j}) \in L^2(\mathbb{R}^3, \mathbb{R})^N\Big\}, \mathbb{R} \right),
\end{align}
which complexifies naturally to
$\big(L^2(\mathbb{R}^3, \mathbb{C})^N\big)^2$ by allowing $U_{\rm r}$ and $U_{\rm j}$ to
be complex-valued.

A great advantage of this representation is that many operators are
then immediate to write in block-matrix form.
For instance, if $A$ is a $\C$-linear operator on
$((L^{2}(\R^{3},\C))^{N},\C)$, then
\begin{align}
  \label{eq:19}
  A(U_{\rm r}+\jj U_{\rm j}) = A U_{\rm r} + \jj A U_{\rm j}\simeqcasida
  \begin{pmatrix}
    A(U_{\rm r}+\ii U_{\rm j})\\ \overline A (U_{\rm r}-\ii U_{\rm j}) 
  \end{pmatrix}
\end{align}
and so, by \eqref{eq:18}:
\begin{align}
  A \simeqcasida
  \begin{pmatrix} A&0\\0&\overline A  \end{pmatrix}
\end{align}
whereas its representation in the Real/Imaginary formalism is:
\begin{align}
  A \simeqreim
\begin{pmatrix} A_{\rm r} & -A_{\rm j} \\ A_{\rm j} & A_{\rm r}\end{pmatrix}
\end{align}
On the other hand, if $L U = A U + B \overline U$ is only $\mathbb{R}$-linear on $((L^{2}(\R^{3},\C))^{N},\C)$, its form in the Casida representation is simply:
\begin{align}
  L \simeqcasida \begin{pmatrix} A & B\\\overline{B}&\overline{A}  \end{pmatrix}
\end{align}
but its representation in the Real/Imaginary formalism is more cumbersome.


Table \ref{table_representations} (Appendix \ref{sec:appendix_tables}) summaries the representations of
some vectors and operators in $\complexvs$.

\paragraph{Equivalence between the formulations of the problem}
The space $\complexvs$ is the abstract complexification of
$\left((L^2(\mathbb{R}^3, \mathbb{C}))^N, \mathbb{R}\right)$, which is
isomorphic to $(L^2(\mathbb{R}^3, \mathbb{C}))^{2N}$, with both the
Casida and real/imaginary representations providing a particular
isomorphism. For simplicity of notations, and without choosing a
particular representation, we will make the identification
\begin{align}
  \complexvs = (L^2(\mathbb{R}^3, \mathbb{C}))^{2N},
\end{align}
which we abbreviate to $(L^2)^{2N}$, and not use the letter $\complexvs$ anymore.
In a similar manner, we denote $(H^2)^{2N}$ the complexified space of
$\left((H^2(\mathbb{R}^3, \mathbb{C}))^N, \mathbb{R}\right)$,
$(L^2_{\alpha})^{2N}$, $(H^2_{\alpha})^{2N}$ their weighted
counterparts, etc.

The operators $\mathcal{M}_{\rm dyn}, \mathcal{K}, \mathcal{S}$, real-linear
operators on $(L^{2}(\R^{3}, \C))^{N}$ naturally become $\mathbb{C}$-linear
operators on $(L^2)^{2N}$, which we denote by straight letters:
\begin{align}
  \label{eq:20}
  \mathcal{M}_{\rm dyn} \mapsto M_{\rm dyn}, \quad \mathcal{K}_{0} \mapsto K_{0}, \quad \mathcal{S_{0}} \mapsto S_{0}.
\end{align}
The
non-linear operator $\mathcal{R}(\cdot, \varepsilon, t)$ can also be
expressed as a non-linear operator $R(\cdot, \varepsilon, t)$ on
$(L^2)^{2N}$. This gives a reformulation of equation
\eqref{eq:time_evolution_whole_system} in
${\mathcal{C}^1\Big([0, T], \big(L^2\big)^{2N}\Big)} \cap {\mathcal{C}^0\Big([0, T], \big(H^2\big)^{2N}\Big)}$:
\begin{align}
  \label{eq:time_evolution_splitted}\boxed{
    \begin{cases}J \partial_t U(t) &= M_{\rm dyn}U(t)+ f(t) V_P\Psi^0 +R(U(t), \varepsilon, t) \\
      U(0)&=0  \end{cases}}
\end{align}
A priori, this equation yields solutions in $\complexvs = (L^2(\mathbb{R}^3, \mathbb{C}))^{2N}$ that do not
necessarily correspond to solutions in $\realvs = (L^2(\mathbb{R}^3, \mathbb{C}))^{N}$, and therefore to
solutions of our original problem
\eqref{eq:time_evolution_whole_system}. However, the equation
preserves the subspace $\realvs$. Therefore, if we establish uniqueness
of solutions of \eqref{eq:time_evolution_splitted}, they
automatically belong to $\realvs$ for all times, and so they solve our
original problem. In the remainder of this paper, we will only
consider \eqref{eq:time_evolution_splitted} and not go back to
\eqref{eq:time_evolution_whole_system} explicitly.



\subsection{Spectrum of \texorpdfstring{$M_{\rm dyn}$}{Mdyn}} 

We have seen that the operator $M_{\rm dyn}$ introduced in \eqref{eq:def_M_dyn} and Section \ref{sec:space_splitting} can be decomposed in $(L^2)^{2N}$ into:
\begin{align}
    M_{\rm dyn}&= \Omega+ K_{0}\ \\
    \Omega&= H_0-\Lambda. 
\end{align}
In particular,
$\sigma(\Omega) = \bigcup_{i=1}^{N} \sigma(H_{0}) - \lambda_{i}$: the
operator $\Omega$ contains the excitation energies of a hypothetical
non-interacting system with Hamiltonian $H_{0}$. As we will see later, some of these
excitations are ``hidden'', and the singularities of
$\widehat \chi$ are only related to the eigenvalues of $M$, where
\begin{align}
  \label{eq:21}
  M = (1-P_0) M_{\rm dyn} (1-P_0)
\end{align}
The spectrum of $(1-P_{0}) \Omega (1-P_{0})$ is only composed of differences
of energies between the unoccupied and occupied parts of the spectrum
of $H_{0}$.
In particular, it has essential spectrum starting at the ionization
threshold $-\lambda_{N}$. A representation of the spectrum of $\Omega$ and
$M$ is given Figure \ref{fig:spectrum_omega}. 
\begin{proposition}
    $M_\mathrm{dyn}$ is a self-adjoint operator acting on $(L^2)^{2N}$ with domain $(H^2)^{2N}$.
    Moreover, $\Omega$ and $M_\mathrm{dyn}$ have the same essential spectrum.
\end{proposition}

\begin{proof}
    It is sufficient to prove that $K_0$ is symmetric on $(L^2)^{2N}$ and that $K_0$ is $\Delta$-compact.
    Let $U,V \in (L^2)^{2N}$, then 
    \begin{align}
        \langle V, K_0 U \rangle_{(L^2)^{2N}} &= \langle V, \big(S_0(U) * \frac{1}{|\cdot |} + v_\mathrm{xc}(\rho) S_0(U) \big) \Psi_0 \rangle_{(L^2)^{2N}} \\
        &= \langle S_0(V), S_0(U) * \frac{1}{|\cdot |} + v_\mathrm{xc}(\rho) S_0(U) \rangle_{(L^2)^2} \\
        &= \langle K_0(V), U \rangle_{(L^2)^{2N}},
    \end{align}
    where we have used that $v_\mathrm{xc}$ is real-valued, that the convolution with $\frac{1}{|\cdot |}$ is symmetric and that for $u \in (L^2)$, we have $\langle V, u \Psi_0 \rangle_{(L^2)^{2N}} = \langle S_0(V), u \rangle_{(L^2)^2}$.

    By Assumption \ref{assumption1}, $V_\mathrm{ext} \in L^2+L^\infty_\varepsilon$. 
    By the decay of $\psi_i$, $\rho * \frac{1}{|\cdot|} \in L^2+L^\infty_\varepsilon$.
    By Assumption \ref{assumption2}, $v_\mathrm{xc}$ is also $L^2+L^\infty_\varepsilon$. 
    Thus $V_\mathrm{tot} = V_\mathrm{ext}+\rho * \frac{1}{|\cdot|}+v_\mathrm{xc}$ is $\Delta$-compact.
    From Lemma~\ref{lem:S_et_K}, $K_0$ is $\Delta$-compact, thus $\Omega$ and $M_\mathrm{dyn}$ have the same essential spectrum.
\end{proof}


$\Omega$ and $M_{\rm dyn}$ have the same essential spectrum but their eigenvalues may differ.
As we will show, the eigenvalues embedded in the essential spectrum of $\Omega$
generically appear as resonances for $M_{\rm dyn}$.




\begin{figure}[!h]
    \centering
    \includegraphics*[width=0.8\textwidth]{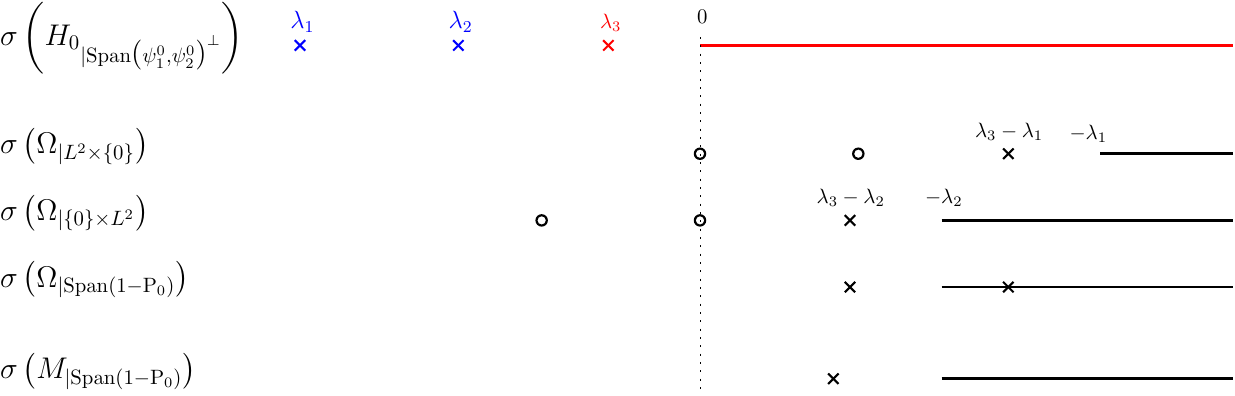}
    \caption{A representation of the spectrum of $\Omega$ and $M$ on
      ${\rm Ran}(1-P_0)$ for $N=2$. The spectrum of $H[\rho_0]$ is
      represented on the first line, with the occupied eigenvalues in
      blue and the energies of the unoccupied states in red. The
      second and third lines represent the spectrum of $H[\rho_0]-\lambda_i$
      $(i=1,2)$, \textit{i.e.} the spectrum of $\Omega$ on each sector of the
      space. The circles correspond to eigenvectors in
      ${\rm Ran}(P_0)$, the ``hidden'' excitations. The spectrum of $\Omega$ restricted to ${\rm Ran}(1-P_0)$ is the combination of the shifted spectra, on the fourth line. The last line shows the spectrum of $M$ restricted to ${\rm Ran}(1-P_0)$.  } \label{fig:spectrum_omega}
\end{figure}

\section{Main results} \label{sec_main_results}



\subsection{Well-posedness and linear response}
Our first result is the well-posedness of
\eqref{eq:time_evolution_splitted} (and therefore of \eqref{eq:TDDFT}):
\begin{theorem}[Well-posedness of \eqref{eq:time_evolution_splitted},
  linear response]~\label{theorem1}
  Under Assumptions 1 to 4, for any $T>0$, there exists
  $\varepsilon_{0}>0$ such that, for all
  $0 < \varepsilon \le \varepsilon_{0}$, 
  \begin{itemize}
\item  The equation
    \eqref{eq:time_evolution_splitted} admits a unique solution $U$ in
    $\mathcal C^{1}([0,T], (L^{2})^{2N}) \cap \mathcal C^{0}([0,T], (H^{2})^{2N})$.

    
  \item We have
    \begin{align}
      \rho(t)& = \rho_{\Psi^{0}(t)+\varepsilon U(t)} = \rho_0 + \varepsilon \int_{0}^{t}\chi(t-t') V_{P} f(t') dt' + \varepsilon^{2} r(t),
    \end{align}
    with $\|r(t)\|_{H^{2}} \le C_{T}$ for all $t \in [0,T]$ for some $C_{T}$
    independent of $t$ and $\varepsilon$.
    The linear response function of the system $\chi(t)$, $t \in \R$ is a uniformly bounded operator from $H^2$ to $H^2$, and given by 
    \begin{align}
    \label{eq:chi_real_time}
    \boxed{
        \chi(t) V_P = - \theta(t) S_{0} \left( e^{-tJM} J (1-P_0)V_P \Psi^0\right)}
    \end{align}
for all $V_P \in H^2$, with $\theta$ being the Heaviside function.
\end{itemize}
  \end{theorem}
  It also results from the estimates in the theorem that $\chi(t)$ defines a
  tempered distribution with values in ${\mathcal L}(H^{2}, H^{2})$,
  with distributional time Fourier transform given by
  \begin{align}
  \label{eq:chi_frequency}
    \hat{\chi}(\omega) V_P&= \lim_{\eta \rightarrow 0^+} -S_{0}\left(\frac 1 {M + \ii (\omega+\ii \eta) J}(1- P_0)V_{P} \Psi^{0} \right).
  \end{align}

  The proof of this theorem is given in Section \ref{sec_well-posedness}. It is based on
a study of the linearized equation, using the coercivity result in
Proposition \ref{proposition:operator_M_dyn} to define the exponential
$e^{-tJM}$, as well as a fixed-point argument in the $H^{2}$ topology.

\begin{remark}[Dyson equation in TDDFT]
    To see the connection between~\eqref{eq:chi_real_time}
    or~\eqref{eq:chi_frequency} with the TDDFT equations in the Dyson formalism, one notices that the operator 
    \(
      \left\lbrace  
      \begin{aligned}
        H^2 & \to (H^2)^{2N} \\
        V_P &\mapsto V_P\Psi_0
      \end{aligned}
    \right.
    \)
    is the $L^2$ adjoint of $S_0$.
    Hence $K_0(U) = S_0^*\Big( \frac{d v_{\mathrm{hxc}}}{d\rho} S_0(U) \Big)$.  
    Thus in real time, introducing $P_0^\perp = 1-P_0$, using a Duhamel formula on $P_0^\perp M P_0^\perp = P_0^\perp \Omega P_0^\perp + P_0^\perp K_0 P_0^\perp$, we have 
    \begin{align}
        \chi(t)V_P &= -\theta(t)S_{0}\left( \Big(e^{-tJ P_0^\perp\Omega P_0^\perp}J - \int_0^t e^{-(t-s)JP_0^\perp\Omega P_0^\perp} JK_0 e^{-sJP_0^\perp( \Omega+K)P_0^\perp} \, \mathrm{d}s\Big) J  P_0^\perp V_P \Psi^0\right)\\
        & = \chi_0(t) V_P + S_0\left( \Big(\int_\R \theta(t-s) e^{-(t-s)JP_0^\perp\Omega P_0^\perp} J S_0^* \frac{dv_{\mathrm{hxc}}}{d\rho} S_0 \theta(s) e^{-sJP_0^\perp( \Omega+K)P_0^\perp} \, \mathrm{d}s\Big) J  P_0^\perp S_0^* V_P \right) \\
        & = \chi_0(t) V_P + \big(\chi_0 \star \frac{dv_{\mathrm{hxc}}}{d\rho} \chi\big)(t) V_P,
    \end{align}
    with
    $\chi_0(t) = -\theta(t)S_{0}\Big(e^{-tJ P_0^\perp\Omega P_0^\perp}J P_0^\perp V_P \Psi_0\Big)$
    which is the linear response function of the non-interacting evolution with $\Omega$ and $\star$ is the convolution in time.
\end{remark}

\subsection{Resonances}
Our second result states that any eigenvalue embedded in the
continuous spectrum of $\Omega$ turns into a resonance in the weakly
interacting regime. We need the following extra assumptions:
\begin{assumption}[Exponential decay of the total potential]\label{assumption6}
    The operator of multiplication by the total potential
    \begin{align}
      \label{eq:22}
      V_{\rm tot}= V_{\rm ext} + \left(\rho_0*\tfrac{1}{|r|}\right)+v_{\rm xc}(\rho_0)
    \end{align}
    is bounded from $H^{2}_{\beta}$ to $L^{2}_{-\beta}$ for some
    $\beta > 0$.
\end{assumption}
This assumption allows to analytically continue the resolvent beyond
the real axis in the topology of exponentially localized functions. It
is justified for atomic systems, where the rotational symmetry implies
the exponential decay of $V_{\rm ext} + \left(\rho_0*\tfrac{1}{|r|}\right)$. It does not hold for general molecules, where the algebraic
decay of the total potential is determined by the first non-zero
moment of the total charge distribution. There, it might be possible
to generalize it in two different directions. First, with less decay
assumptions, we should be able to obtain a weaker notion of resonance
\cite{orth1990quantum}. Second, with minimal decay assumptions but
assuming analyticity at infinity of $V_{\rm tot}$, it should be
possible to obtain similar results as the ones in this paper \cite{barry_dilatations} using the
theory of analytic dilatations.

\begin{assumption}[Excitation energy embedded in continuous spectrum]\label{assumption7} 
  Let $\lambda_{i_0}$ and $\lambda_{a_0}$ two eigenvalues of $H_0$, with
  ${i_0} \in \{1, ..., N\}$ an occupied orbital and ${a_0}>N$ an unoccupied
  orbital (eigenvector of $H_{0}$). We assume
  \begin{itemize}
  \item $\lambda_{{a_0}}-\lambda_{{i_0}} > -\lambda_{N}$ (excitation embedded in
    continuous spectrum)
  \item $\lambda_{{a_0}} - \lambda_{{i_0}} \neq -\lambda_{i}$ for any
    $i=1, \dots, N$ (excitation not at ionization thresholds)
  \item $\lambda_{{a_0}}$ and $\lambda_{{i_0}}$ are simple eigenvalues.
  \end{itemize}
\end{assumption}
The last assumption is not crucial, and is made to simplify the final
expression of the decay rate.


\begin{theorem}\label{th:resonances}

  Under Assumptions 1 to 7, there exists $\alpha_{0}>0$ such that, for all $0<\alpha < \alpha_{0}$, 
  $\widehat \chi$ admits a meromorphic continuation as an
  operator from $L^{2}_{\alpha}$ to $L^{2}_{-\alpha}$ from the upper
  complex plane to a complex neighborhood of $\lambda_{{a_0}}-\lambda_{{i_0}}$.

  If
  $\delta = \|K_0\|_{\mathcal{L}\left((L^2_\alpha)^{2N}, (H^2_{-\alpha})^{2N}\right)}$
  is small enough and $\psi_{i_0}\overline{\psi}_{a_0}+\psi_{a_0}\overline{\psi}_{i_0} \not=0$, this continuation admits a single pole at a location $z_{\rm pole}$ with
  \begin{align}
    \mathrm{Re}(z_{\rm pole})=\lambda_{a_0}-\lambda_{i_0} + \Delta E + O(\delta^{3}), \quad \text{and} \quad \mathrm{Im}(z_{\rm pole}) = -\Gamma+O(\delta^3),
  \end{align}
  where $\Delta E$ and $\Gamma$, of second order in $\delta$, are given by a Fermi golden rule type
  expression (see \eqref{eq:13}). In particular, $\Gamma$ is
  non-negative and generically non-zero.


\end{theorem}
This pole can be interpreted in several ways. In the linear response
function $\widehat \chi(\omega)$ for $\omega$ real, it corresponds to
a Breit-Wigner bump at frequency ${\rm Re}(z_{\rm pole})$ and with
width $-{{\rm Im}(z_{\rm pole})}$; see for instance
\cite{toulouse}. Dynamically, it corresponds to a long-lived state with
lifetime $-\frac{1}{{\rm Im}(z_{\rm pole})}$; see for instance
\cite{dyatlov_mathematical_2019,orth1990quantum}.


\section{Well-posedness and linear response} \label{sec_well-posedness}
In this Section, we prove Theorem \ref{theorem1}.

\subsection{The linearized equation}\label{sec:linearized_equation}
We will do this by
studying the linearized equation
\begin{align}
  \label{eq:23}
  J \partial_{t} U &= M_{\rm dyn} U, \quad U(0) = U_{0}.
\end{align}
\begin{lemma}[Well-posedness of the linearized equation]\label{lemma_propagator}~
  Suppose that $e^{-JMt}$ defines a uniformly bounded semigroup in $(H^2)^{2N}$, then
  for all $U_{0} \in (H^{2})^{2N}$, there exists a unique solution
  \begin{align}
    \label{eq:24}
    U(t) = e^{-JM_{\rm dyn} t} U_{0}
  \end{align}
  of \eqref{eq:23}, satisfying
  \begin{align}
    \label{eq:25}
    \|U(t)\|_{(H^{2})^{2N}} \le C_{\rm dyn} (1+t^2) \|U_{0}\|_{(H^{2})^{2N}}
  \end{align}
  for some $C_{\rm dyn} > 0$ and for all $t \in \R$.

  If furthermore $U_{0} \in {\rm Ran}(1-P_{0})$, then
  \begin{align}
    \label{eq:27}
    \|U(t)\|_{(H^{2})^{2N}} \le C \|U_{0}\|_{(H^{2})^{2N}}
  \end{align}
  for some $C > 0$ and for all $t \in \R$.
\end{lemma}
\begin{proof}
  
To prove this lemma, we look at the structure of orbital variations
$U$. Any orbital variation $U \in \realvs = (L^{2}(\R^{3}, \C)^{N}, \R)$ can
be decomposed orthogonally as
\begin{align}
  \label{eq:28}
  U = U_{\rm S} + U_{\rm A} + U_{\perp},
\end{align}
where
\begin{align}
  \label{eq:29}
  U_{\rm S} = \Psi^{0} S, \quad U_{\rm A} = \Psi^{0} A, \quad U_{\perp} \in {\rm Ran}(1-P_{0})
\end{align}
with $S$ a $N \times N$ Hermitian matrix (explicitly:
$S=S_{\rm r}+\jj S_{\rm j}$ satisfies
$S_{\rm r}^{T} = S_{\rm r}, S_{\rm j}^{T} = -S_{\rm j}$) and $A$ a skew-Hermitian matrix.
This corresponds to an orthogonal splitting of the space $Y_{\R}$ as
\begin{align}
  \label{eq:30}
  Y_{\R} = Y_{\R,{\rm S}} + Y_{\R,{\rm A}} + Y_{\R, \perp}
\end{align}
which induces an orthogonal splitting of $Y = (L^{2})^{2N}$.
\begin{align}
  Y = Y_{\rm S} + Y_{\rm A} + Y_{\perp}.
\end{align}

The variations in $Y_{\rm S}$ are ``growth modes'': they violate the
normalization condition $\Psi^{*} \Psi = 1$, whose tangent space is
$Y_{\rm A} + Y_{\perp}$. This normalization condition is preserved by
the flow of the nonlinear equation, and therefore its tangent space is
preserved by the linearized equation:
\begin{align}
  \label{eq:33}
  (-J M_{\rm dyn} U)_{\rm S} = 0 \quad \forall U \in Y_{\rm A} + Y_{\perp}.
\end{align}
The variations in
$Y_{\rm A}$ are ``gauge modes'', that correspond to a rotation of the
orbitals between themselves and therefore produce no observable
physical output
\begin{align}
  \label{eq:31}
  &S_{0}(U_{\rm A}) = 0 \quad  \forall U_{\rm A} \in Y_{\rm A},
\end{align}
so that in particular $K_{0} Y_{\rm A} = 0$. By self-adjointness,
${\rm Ran}(K_{0}) \perp Y_{\rm A}$.
Finally, for any matrix $Q$
\begin{align}
  \label{eq:34}
  \Omega (\Psi^{0} Q) = \Psi^{0} Q', \quad \text{ with } Q'_{ij} = (\lambda_{i}-\lambda_{j})Q_{ij},
\end{align}
so that we have the following sparsity patterns:
\begin{align*}
  J =  \begin{pmatrix}
    0&J&0\\
    J&0&0\\
    0&0&J
  \end{pmatrix}, \quad
  \Omega = 
  \begin{pmatrix}
    0&\ast&0\\
    \ast&0&0\\
    0&0&\ast
  \end{pmatrix}, \quad
K_{0} =
  \begin{pmatrix}
    \ast&0&\ast\\
    0&0&0\\
    \ast&0&\ast
  \end{pmatrix}.
\end{align*}

Altogether, \eqref{eq:23} can be rewritten as
\begin{align}
  \label{eq:36}
  \partial_{t} U &=
  \begin{pmatrix}
    -J \Omega&0&0\\
    L_{\rm AS}&-J \Omega&L_{A\perp}\\
    L_{\rm \perp S}&0&-JM
  \end{pmatrix} U,
\end{align}
where $L_{\rm AS}, L_{\rm A\perp}$ and $L_{\perp S}$ are blocks of
$-J K_{0}$, whose expression is not relevant (they are bounded, since
they have either a starting or ending space of finite dimension). The
solution can therefore be obtained formally as
\begin{align*}
  U_{\rm S}(t) &= e^{-J \Omega t} U_{\rm S}(0)\\
  U_{\perp}(t) &= e^{-JMt} U_{\perp}(0) + \int_{0}^{t} e^{-JM(t-t')} L_{\perp \rm S} U_{\rm S}(t') dt'\\
  U_{\rm A}(t) &= e^{C t} U_{\rm A}(0)  +\int_{0}^{t} e^{-J\Omega(t-t')} (L_{\rm AS} U_{\rm S}(t') + L_{\rm A \perp} U_{\perp}(t')) dt'
\end{align*}
Since
\begin{align*}
  J \Omega \simeqcasida
  \begin{pmatrix}
    \ii \Omega &0\\0&-\ii\Omega
  \end{pmatrix},
\end{align*}
$e^{-J\Omega t}$ is unitary on the finite dimensional spaces $Y_{\rm S}$
and $Y_{\rm A}$ in the $L^2$ topology for all
$t \in \R$,
and therefore $e^{-J\Omega t}$ is bounded uniformly in time in the
$H^{2}$ topology. Our task is therefore now to give a sense and obtain a uniform bound in time for
$e^{-JMt}$ in the $H^2$ topology of $Y_{\perp}$, from which Lemma
\ref{lemma_propagator} follows immediately.
\end{proof}

The main difficulty to study $e^{-JMt}$ is that $-JM$ is not
skew-adjoint because $J$ does not commute with $K_{0}$. This prevents
us from using the self-adjoint functional calculus. However,
since $M$ is positive, we can make use of the following formal equality
\begin{align*}
  e^{-JMt} = M^{-1/2}e^{-M^{1/2} J M^{1/2} t}M^{1/2}
\end{align*}
where now $M^{1/2} J M^{1/2}$ is skew-adjoint. Lemma
\ref{lemma_propagator} then follows from the following technical lemma:

\begin{lemma}[Properties of $e^{-tM^{1/2}JM^{1/2}}$]~
  \label{lemma:boundedness_propagator}
  
    \begin{enumerate}
        \item $x \mapsto \|Mx\|_{L^2}$ is an equivalent norm to $H^2$ on $\mathrm{Ran}(1-P_0)$;
%
      \item For any $t \in \R$, $e^{-tM^{1/2}JM^{1/2}}$ is unitary on $(L^2)^{2N}$;
      \item For $t \in \R$, $M^{-1/2}e^{-tM^{1/2}JM^{1/2}}M^{-1/2}$ as an operator on $(H^2)^{2N}$ is uniformly bounded, \emph{i.e.} there is a constant $C>0$ such that for all $t \in \R$
      \[
        \big \|M^{-1/2}e^{-tM^{1/2} J M^{1/2}}M^{1/2}\big \|_{\mathcal{L}\left((H^2)^{2N},(H^2)^{2N}\right) } \leq C;
       \] 
      \item  For any $\Phi_0 \in (H^2)^{2N}$, let $\Phi(t) = M^{-1/2}e^{-tM^{1/2}JM^{1/2}}M^{1/2}\Phi_0$. 
       Then $\Phi \in C^0(\mathbb{R},(H^2)^{2N}) \cap C^1(\mathbb{R},(L^2)^{2N})$ solves $\Phi'(t) =- JM \Phi(t)$.
    \end{enumerate}
\end{lemma}
\begin{proof}[Proof of Lemma \ref{lemma:boundedness_propagator}]~
\begin{enumerate}

    \item \textbf{$M$ induces a norm equivalent to the $H^2$ norm on ${\rm Ran}\left(1-P_0\right)$}
    
  Let us first notice that $K_0$ is bounded from $(H^2)^{2N}$ to itself.  By Lemma \ref{lemma_stabilite}, $\Uperp \mapsto \left(S_0(\Uperp)* \frac{1}{|r|}\right)\Psi^0 \in \mathcal{L}(H^2, H^2)$, and $v_{\rm xc}'(\rho_0)$ is bounded because $v_{\rm xc}'$ is continuous.

    Let $\Uperp$ in $\big(H^2\big)^{2N}\cap {\rm Ran}\left(1-P_0\right)$.
    \begin{align}
        \|\Uperp\|_{H^2}&= \|\Uperp\|_{L^2}+ \|\Delta \Uperp\|_{L^2}\\
        \| M \Uperp\|_{L^2}&= \| M(-\Delta+1)^{-1}(-\Delta+1)\Uperp\|_{L^2}\\
        &\leq \| M (-\Delta+1)^{-1}\|_{\mathcal{L}\left((L^2)^{2N}, (L^2)^{2N}\right)} \|\Uperp\|_{H^2}
    \end{align}

    $(\Omega + K_0)(-\Delta+1)^{-1}$ is bounded because $V_\mathrm{tot}+K_0$ is a $\Delta$-compact operator.
    Therefore the $H^2$ norm dominates $ M $ on ${\rm Ran}\left(1-P_0\right)$.
    We now prove the reverse statement. We notice that $V_\mathrm{tot}+K_0$ is a $\Delta$-compact operator, thus $\Delta$-infinitesimally bounded:
    \begin{align}
        \forall \eta >0, ~\exists C_{\eta}>0 ~&:~ \forall \Uperp \in \big(H^2\big)^{2N}\cap {\rm Ran}\left(1-P_0\right), \\
         \|(V_\mathrm{tot}+K_0)\Uperp\|_{L^2} &\leq \eta \|\Delta \Uperp\|_{L^2}+C_{\eta}\|\Uperp\|_{L^2} \\
        \| M \Uperp\|_{L^2}& \geq \|(-\frac{1}{2} \Delta -\Lambda)\Uperp\|_{L^2} -\|(V_\mathrm{tot}+K_0)\Uperp\|_{L^2}\\
        \| M \Uperp\|_{L^2}& \geq (\tfrac{1}{2}-\eta) \|\Delta \Uperp\|_{L^2}- (C_{\eta}+ \|\Lambda\|_{\mathcal{L}\left((L^2)^{2N} , (L^2)^{2N}\right)})\|\Uperp\|_{L^2} \label{eq:domination_1}
    \end{align}
    Furthermore, $ M $ is $L^2$-coercive: 
    \begin{align}
        \exists c >0 , ~ \forall \Uperp \in \big(H^2\big)^{2N}\cap {\rm Ran}\left(1-P_0\right), ~ \| M  \Uperp\|_{L^2} \geq c\|\Uperp\|_{L^2}  \label{eq:domination_2}
    \end{align}
    
    Combining equations \eqref{eq:domination_1}-\eqref{eq:domination_2} allows to conclude that $M$ dominates the $H^2$ norm on ${\rm Ran}\left(1-P_0\right)$. Thus the $M$ norm and the $H^2$ norm are equivalent on $\mathrm{Ran}(1-P_0)$.

%
%
   
    \item \textbf{$e^{-M^{1/2} J M^{1/2}}$ is unitary on $\big(L^2\big)^{2N}$ }
    
    
    $M^{1/2} J M^{1/2}$ is a skew-adjoint operator on $\big(L^2\big)^{2N}$, with domain $\big(H^2\big)^{2N}$. We can exponentiate it, and:
    \begin{align}
        e^{-M^{1/2} J M^{1/2}} \in \mathcal{L}\left(\big(L^2\big)^{2N},\big(L^2\big)^{2N}\right)
    \end{align}
    $e^{-M^{1/2} J M^{1/2}}$ is a unitary operator (and for any $t \in \mathbb{R}$,  $e^{-tM^{1/2} J M^{1/2}}$ is unitary).

    \item \textbf{$M^{-1/2}e^{-tM^{1/2} J M^{1/2}}M^{1/2}$ is uniformly bounded in time from $\big(H^2\big)^{2N}$ to itself}

    The proof is carried out in two steps. First we show that the $H^1$ norm is equivalent to the norm induced by $M^{1/2}$ on  ${\rm Ran}\left(1-P_0\right)$. Let $\Uperp$ in $(H^1)^{2N} \cap  {\rm Ran}\left(1-P_0\right)$:
    \begin{align}
        \|M^{1/2} \Uperp\|_{L^2}^2 &= \langle \Uperp, M  \Uperp\rangle_{L^2} \\
        &= \langle \Uperp, (-\frac{1}{2}\Delta- \Lambda + V_\mathrm{tot} + K) \Uperp \rangle_{L^2} \\
        &= \langle (-\frac{1}{2}\Delta- \Lambda)^{1/2}\Uperp, (1 + (-\frac{1}{2}\Delta- \Lambda)^{-1/2}(V_\mathrm{tot} + K)(-\frac{1}{2}\Delta- \Lambda)^{-1/2}) (-\frac{1}{2}\Delta- \Lambda)^{1/2}\Uperp \rangle_{L^2} \\
        &= \frac{1}{2} \langle \nabla \Uperp, \nabla \Uperp \rangle_{L^2} + \langle \Uperp, -\Lambda \Uperp \rangle_{L^2} \\
        & \qquad \qquad+ \langle (-\frac{1}{2}\Delta- \Lambda)^{1/2}\Uperp, (-\frac{1}{2}\Delta- \Lambda)^{-1/2}(V_\mathrm{tot} + K)(-\frac{1}{2}\Delta- \Lambda)^{-1/2})(-\frac{1}{2}\Delta- \Lambda)^{1/2}\Uperp \rangle.
    \end{align}
    Since $((-\frac{1}{2}\Delta- \Lambda)^{-1/2}(V_\mathrm{tot} + K)(-\frac{1}{2}\Delta- \Lambda)^{-1/2})^2$ is compact on $L^2$ and $(V_\mathrm{tot} + K)(-\frac{1}{2}\Delta- \Lambda)^{-1/2})^2$ is bounded from $(L^2)^{2N}$ to $(L^2)^{2N}$ by Lemma~\ref{lem:S_et_K}, then $(-\frac{1}{2}\Delta- \Lambda)^{-1/2}(V_\mathrm{tot} + K)(-\frac{1}{2}\Delta- \Lambda)^{-1/2}$ is compact on $L^2$.
    Thus, for any $\eta>0$, there is $C_\eta>0$, such that for all $U \in (L^2)^{2N} \cap \mathrm{Ran}(1-P_0)$
    \[
        \big| \langle U, (-\frac{1}{2}\Delta- \Lambda)^{-1/2}(V_\mathrm{tot} + K)(-\frac{1}{2}\Delta- \Lambda)^{-1/2})U \rangle_{L^2} \big| \leq \eta \langle U, U \rangle_{L^2} + C_\eta \langle (-\frac{1}{2}\Delta- \Lambda)^{-1/2}U, (-\frac{1}{2}\Delta- \Lambda)^{-1/2}U \rangle_{L^2}.
    \]
    Plugging in the previous equation, we have 
    \[
        \|M^{1/2} \Uperp\|_{L^2}^2 \leq \left(\frac{1}{2}-\eta\right) \Big( \langle \nabla \Uperp, \nabla \Uperp \rangle_{L^2} + \langle \Uperp, -\Lambda \Uperp \rangle_{L^2} \Big) - C_\eta \langle \Uperp, \Uperp \rangle_{L^2}.
    \]
%
    Using this equation with $\eta$ small enough, and the $L^2$ coercivity of $M^{1/2}$, the same reasoning as before allows to conclude that $M^{1/2}$ is equivalent to the $H^1$ norm on ${\rm Ran}\left(1-P_0\right)$. 
     

    We have seen that $M$ induces a norm equivalent to $H^2$ and $M^{1/2}$ to $H^1$. Since $J$ is unitary on $H^1$, $M^{1/2}JM^{1/2}$ induces a norm equivalent to the $H^2$ norm. Since it commutes with $e^{-tM^{1/2}JM^{1/2}}$,  $e^{-tM^{1/2}JM^{1/2}}$ is a uniformly bounded operator from $(H^2)^{2N}$ to  $(H^2)^{2N}$.

    \item \textbf{$M^{-1/2}e^{-M^{1/2} J M^{1/2} t}M^{1/2}$ is the semigroup associated to $\Phi' = -JM \Phi$}
  
    We first prove that $e^{-M^{1/2} J M^{1/2} t}$ is a bounded operator on $(H^1)^{2N}$.
    Since $e^{-tM^{1/2} J M^{1/2}}$ is bounded as an operator on $L^{2}$
    and $H^{2}$, by a Riesz-Thorin type interpolation argument \cite[Theorems 7.1,
    5.1]{Lions1972}, we have the $H^{1}$ bound
    \begin{align}
      \Big\|e^{-tM^{1/2} J M^{1/2}}\Big\|_{\mathcal{L}\left((H^1)^{2N}, (H^1)^{2N}\right)}\leq  \Big\|e^{-tM^{1/2} J M^{1/2}}\Big\|_{\mathcal{L}\left((L^2)^{2N}, (L^2)^{2N}\right)}^{1/2}\Big\|e^{-tM^{1/2} J M^{1/2}}\Big\|_{\mathcal{L}\left((H^2)^{2N}, (H^2)^{2N}\right)}^{1/2}.
    \end{align}
    Now for $\Phi_0 \in (H^2)^{2N}$, let $\Phi(t) = M^{-1/2}e^{-M^{1/2} J M^{1/2} t}M^{1/2}\Phi_0$.
    Then $$\Phi'(t) = -JM M^{-1/2}e^{-M^{1/2} J M^{1/2} t}M^{1/2}\Phi_0 = -JM \Phi(t),$$
    where we have used that $M^{1/2} : (H^1)^{2N} \to (L^2)^{2N}$, $M^{-1/2} : (H^1)^{2N} \to (H^2)^{2N}$ are bounded on ${\rm Ran}(1-P_0)$ and extended to ${\rm Ran}(P_0)$ by the identify.
    \end{enumerate}
\end{proof}






\subsection{Well-posedness of the nonlinear system}

In this section, we prove that equation \eqref{eq:time_evolution_splitted} is well-posed.
By the Duhamel formula, any solution of
\eqref{eq:time_evolution_splitted} is also a solution of
\begin{align}
  \label{eq:Duhamel}
  U(t) &= - \int_0^t e^{-JM_{\rm dyn}(t-t')}J \left(f(t') V_P\Psi^0 + R(U, \varepsilon, t')\right) dt'.
\end{align}
Let
\begin{align}
  \label{eq:7}
  F(\varepsilon,U)(t) = U(t) + \int_0^t e^{-JM_{\rm dyn}(t-t')} J \left( f(t') V_P\Psi^0 + R(U, \varepsilon, t') \right) dt'.
\end{align}
We solve the equation $F(\varepsilon,U)=0$ in the Banach space
$\mathcal{C}^0\left([0, T], \big(H^2\big)^{2N}\right)$ for a fixed
$T$.
We first prove that $F$ is $\mathcal C^{1}$ on
$\R \times \mathcal C^{0}([0,T], (H^{2})^{2N})$. 

Recall that:
\begin{align}
  \label{eq:2}
  {R}(U, \varepsilon, t)&= \varepsilon f(t) V_P  U+ \epsilon^{-1}\left(H[\rho_{\Psi^{0}+\varepsilon U}]- H[\rho_0]\right) \Psi^{0}- K_0(U)+ \left(H[\rho_{\Psi^{0}+\varepsilon U}]- H[\rho_0]\right) U
\end{align}

By assumption on $V_P$ and Lemma~\ref{lem:S_et_K}, $U \mapsto \epsilon f(t)V_{P} U -K_0(U)$ is linear and maps $(H^{2})^{2N}$ to itself.
The mapping $\rho \mapsto H[\rho]- H[\rho_0]=v_{\rm hxc}(\rho)-v_{\rm hxc}(\rho_0)$ is $\mathcal{C}^1$ from $L^1\cap H^2$ to $\mathcal{L}((H^2)^{2N}, (H^2)^{2N})$ because of Assumption \ref{assumption2}, and the density mapping $\Psi \mapsto \rho_\Psi$ is $\mathcal{C}^\infty$ from $(H^2)^{2N}$ to $L^1\cap H^2$. 
This shows that $R$ is $\mathcal{C}^1$ in $U$ in the $(H^{2})^{2N}$ topology. 

Since $e^{-JM_{\rm dyn}t}$ is
bounded as an operator from $(H^{2})^{2N}$ to itself uniformly in
$t \in [0,T]$, $F$ is $\mathcal C^{1}$ on
 $\R \times \mathcal C^{0}([0,T], (H^{2})^{2N})$. 
 
 Furthermore, $\|(H[\rho_{\Psi^0+\varepsilon U}]-H[\rho_0])\Psi^0 - \varepsilon K_0(U)\|_{H^2} = O(\varepsilon^2 \|U\|_{H^2}^2) $ because $v_{xc}$ is $\mathcal{C}^1$. 
Thus $F(\varepsilon,U)(t) = U(t) + \int_{0}^{t} e^{-JM_{\rm dyn}(t-t')} J \left( f(t') V_P\Psi^0 \right) dt' + \mathcal{O}(\varepsilon \|U\|^2_{H^2})$. 
This shows that:
\begin{align}
  \label{eq:5}
  F(0,0) &= 0\\
  \frac{\partial F}{\partial U}(0,0) &= {\rm Id}
\end{align}
We can apply the implicit function theorem to $F$ on
$\mathcal{C}^0\left([0, T], \big(H^2\big)^{2N}\right)$. For $\varepsilon$ small enough, there exists a unique solution
$U(\varepsilon)$ of $F(\varepsilon,U)=0$ in a neighborhood of $0$. Finally, we can check that this solution is in
$\mathcal C^{1}([0,T], (L^{2})^{2N}) \cap \mathcal C^{0}([0,T], (H^{2})^{2N})$ and
solves \eqref{eq:time_evolution_splitted}.

Again by the implicit function theorem, the solution $U(\varepsilon)$
is $\mathcal C^{1}$ from a neighborhood of $0$ to
$\mathcal{C}^0\left([0, T], \big(H^2\big)^{2N}\right)$, and
\begin{align}
  \label{eq:8}
  \frac d {d\varepsilon} U(0) = - \frac{\partial F}{\partial U}(0,0)^{-1} \frac{\partial F}{\partial \varepsilon}(0,0) = -\int_0^te^{-JM_{\rm dyn}(t-t')}J f(t') V_P\Psi^0 dt' = U^{1}
\end{align}
It follows that, for all $t \in [0,T]$,
\begin{align*}
  \rho(t)=\rho_0+\varepsilon \rho^{(1)}(t)+O(\varepsilon^2)
\end{align*}
in $H^{2}$, with $\rho^{(1)}(t)=S_0(U^1(t))$. Following the notations of section \ref{sec:linearized_equation}, $U^1_S=0$ and thus $S_0(U^1(t))=S_0(U^1_{\perp}(t))$. The expression of $U^1_{\perp}$ is given by:
\begin{align}
    U^1_{\perp}(t)= -\int_0^te^{-JM(t-t')}J f(t') V_P\Psi^0 dt'
\end{align}
The expression of $\chi(t)$
follows.

Since $t \mapsto U^1_{\perp}(t)$ is continuous and bounded in the $H^2$ topology, $\chi$
defines a tempered distribution on $\R$ with values in $\mathcal L(H^{2}, H^{2})$. Since
$\chi$ is causal, we have $\lim_{\eta \to 0^{+}} e^{-\eta t} \chi(t) = \chi(t)$
in the sense of tempered distributions. As operators on $H^{2}$, we
have
\begin{align}
  \int_{0}^{\infty} e^{-\eta t} e^{i\omega t} e^{-JM t} &= \frac 1 {JM - \ii(\omega+\ii\eta)}
\end{align}
so that, in the sense of tempered distributions,
\begin{align}
  \hat{\chi}(\omega) V_P&= \lim_{\eta \rightarrow 0^+} -S_{0}\left(\frac 1 {JM - \ii(\omega+\ii\eta)} J V_{P} \psi^{0} \right) = \lim_{\eta \rightarrow 0^+} -S_{0}\left(\frac 1 {M + \ii(\omega+\ii\eta) J} V_{P} \psi^{0} \right).
\end{align}

\section{Resonances}\label{sec:propagator_and_response}
In this section, we prove Theorem \ref{th:resonances} on resonances,
assuming Assumptions \ref{assumption6} (exponential decay of the total
potential) and \ref{assumption7} (non degenerate excitation energy
$\lambda_a-\lambda_i$).

In the Casida representation,
\begin{align}
  \label{eq:3}
  M + \ii zJ \simeqcasida
  \begin{pmatrix}
    \Omega-z&0\\0&\Omega+z
  \end{pmatrix} + K_0
\end{align}
where we recall that
$(\Omega U)_{i} = \left( -\tfrac 1 2 \Delta -\lambda_{i} \right)u_{i} + V_{\rm tot} u_{i}$.
We will show the analytical continuation of the inverse of $M+\ii z J$
from the upper complex to the lower one by a perturbation argument,
starting from the resolvent of the Laplacian. We will then study
resonances by identifying an eigenvalue at the same energy as
continuous spectrum in the operator $\Omega$, and proving that this
generically becomes a resonance when perturbed by $K_{0}$.

\subsection{Analytic continuation of the free Laplacian}
We start with a classical lemma on the analytic continuation of the
free Laplacian.
\begin{lemma}[Meromorphic continuation of the resolvent of the free
  Laplacian]\label{lemma_free_laplacian_direct_version}
  For $\alpha > 0$, the resolvent $(z+\tfrac 1 2 \Delta)^{-1}$ as an
  operator from $L^{2}_{\alpha}$ to $H^{2}_{-\alpha}$
  has an analytic continuation from the first quadrant
  ${\rm Re}(z) > 0, {\rm Im}(z) > 0$ to the region
  ${\rm Re}(z) > 0, {\rm Im}(\sqrt{2 z}) > -\alpha$.
\end{lemma}

This lemma is classical and well-known in the mathematical study of
resonances; see \cite{dolph1966analytic} for instance. Nevertheless,
to keep the paper self-contained, we reprove it here.
\begin{proof}[Proof of Lemma
  \ref{lemma_free_laplacian_direct_version}]~ 

  Let $\varphi$ in $L^2_{\alpha}$ and $\psi$ in the dual space of
  $H^2_{-\alpha}$; in particular, $\psi \in H^{-2}$ and $\widehat \psi$
  is smooth. We will study the
  analytic continuation in $z$ of $\langle \psi | (z+\tfrac{\Delta}{2})^{-1}  \varphi \rangle$
  as ${\rm Im}(z)$ becomes negative.

  We use the usual Fourier transform convention in space, $\mathcal{F}\psi(q)= \tfrac{1}{(2 \pi)^{3/2}}\int_{\mathbb{R}^3} e^{-\ii q \cdot r}\psi(r) dr$. Then for ${\rm Im}(z)>0$, by the Parseval formula:  

  \begin{align*} 
    \langle \psi | (z+\tfrac{\Delta}{2})^{-1} | \varphi \rangle &= \int_{\mathbb{R}^3} \frac{1}{z-\tfrac{1}{2}|q|^2}\overline{\mathcal{F}{\psi}(q)} \mathcal{F}{\varphi}(q)dq
  \end{align*}
    Let $I \subset \R^{+}$ be an interval not touching $0$. We split the above
    integral in two contributions $A_{I}(z)$ and $A_{\R^{+} \setminus I}(z)$, according to
    whether $\tfrac 1 2 |q|^{2}$ belongs to $I$ or not.

    For $A_{\R^{+} \setminus I}(z)$, since $\psi \in H^{-2}$ and $\phi \in L^{2}$, we
    have that
    \begin{align*}
      A_{\R^{+} \setminus I}(z) &= \int_{q \in \mathbb{R}^3, \tfrac 1 2 |q|^{2} \notin I} \frac{1}{z-\tfrac{1}{2}|q|^2}\overline{\mathcal{F}{\psi}(q)} \mathcal{F}{\varphi}(q)dq\\
      &=\int_{q \in \mathbb{R}^3, \tfrac 1 2 |q|^{2} \notin I} \frac{1+|q|^{2}}{z-\tfrac{1}{2}|q|^2}\frac{\overline{\mathcal{F}{\psi}(q)}}{1+|q|^{2}} \mathcal{F}{\varphi}(q)dq
    \end{align*}
    is analytic for $z$ in $I + \ii \R$.

    We now write $A_{I}(z)$ as
    \begin{align*}
      A_{I}(z)&= \int_{I}\frac 1 {z-\lambda} \underbrace{\frac{1}{(2\pi)^3}\int_{S^2}{\overline{\mathcal{F}{\psi}(\sqrt{2\lambda}\tilde{q})} \mathcal{F}{\varphi}(\sqrt{2\lambda}\tilde{q})} \sqrt{2\lambda}d\tilde{q}}_{f(\lambda)} d\lambda .
    \end{align*}
    The functions defined for $\lambda$ real by
    \begin{align*}
      \mathcal{F}\varphi(\sqrt{2\lambda}\tilde{q})&= \int_{\mathbb{R}^3}e^{-\ii \sqrt{2\lambda} \tilde{q} r}\varphi(r) dr = \int_{\mathbb{R}^3}e^{-\ii {\rm Re}(\sqrt{2\lambda}) \tilde{q} r} e^{{{\rm Im}(\sqrt{2\lambda}) \tilde{q} r}}\varphi(r) dr\\
      {\overline{\mathcal{F} \psi\left( {\sqrt{2\lambda}}\tilde{q} \right)}} &= \int_{\mathbb{R}^3}e^{\ii \sqrt{2\lambda} \tilde{q} r}\overline{\psi(r)} dr= \int_{\R^{3}}e^{\ii {\rm Re}(\sqrt{2\lambda}) \tilde{q} r} e^{-{{\rm Im}(\sqrt{2\lambda}) \tilde{q} r}} \overline{\psi(r)} dr
    \end{align*}
    (where the second line is to be understood in the sense of
    distributions) extend analytically for $\lambda$ in the set
    \begin{align*}
      D_{I,\alpha} = \{z \in I+\ii\R, |{\rm Im}(\sqrt{2 z})| < \alpha\}.
    \end{align*}
    It follows that $f$ also extends to this set, so that, for
    $z \in D_{I,\alpha}, {\rm Im}(z) > 0$
    \begin{align*}
      A_{I}(z)&= \int_{I}\frac 1 {z-\lambda} f(\lambda)d\lambda = \int_{C_{\varepsilon}}\frac 1 {z-\lambda} f(\lambda)d\lambda
    \end{align*}
    where $C_{\varepsilon}$ is a contour starting at $\inf I$, dropping vertically
    to the bottom of $D_{I,\alpha-\varepsilon}$, following its lower edge, and coming
    back up at $\sup I$. This shows that $A_{I}$ extends analytically to
    $D_{I,\alpha-\varepsilon}$. Since $I$ and $\varepsilon$ were arbitrary, this concludes the proof.

\end{proof}

\subsection{Meromorphic continuation of the resolvent of
  \texorpdfstring{$JM$}{JM}}

Consider the family of operators defined in the Casida representation by
\begin{align}
  -\tfrac 1 2 \Delta - \Lambda + \ii z J
  \simeqcasida
  \begin{pmatrix}
    -\tfrac 1 2 \Delta - \Lambda  - z& 0\\
    0& -\tfrac 1 2 \Delta - \Lambda  + z
  \end{pmatrix}
\end{align}
From Assumption \ref{assumption7}, when $z$ is close to
$\lambda_{a_{0}}-\lambda_{i_{0}}$, the shifts
$(\lambda_{i} \pm z)$ are close to a nonzero real number.
It follows from the previous Lemma that this family has an inverse
that can be continued analytically from the upper half complex plane
to a complex neighborhood of
$\lambda_{a_{0}} - \lambda_{i_{0}}$ as an operator from $L^{2}_{\alpha}$ to
$H^{2}_{-\alpha}$. We now write formally for ${\rm Im}(z) > 0$
\begin{align}
  M_{\mathrm{dyn}} +\ii z J &= -\tfrac 1 2 \Delta - \Lambda +\ii z J + (V_{\rm tot} + K_0)\\
  \label{eq:M_as_pert_of_Mess}
  \frac 1 {M_{\mathrm{dyn}} +\ii z J} &= \frac 1 {-\tfrac 1 2 \Delta - \Lambda +\ii z J} \left( 1 + (V_{\rm tot} + K_0) \frac 1 {-\tfrac 1 2 \Delta - \Lambda +\ii z J} \right)^{-1}.
\end{align}

This formal computation is justified in the proof of Proposition \ref{proposition_meromorphic_1}
%
\begin{proposition}[Meromorphic continuation of $(M+\ii z J)^{-1}$]\label{proposition_meromorphic_1}
  For $\alpha > 0$ small enough, the operator $(M + \ii z J)^{-1}$
  from $(L^{2}_{\alpha})^{2N}$ to $(H^{2}_{\alpha})^{2N}$
  has an analytic continuation from the upper complex plane to a
  complex neighborhood of $\lambda_{a_{0}}-\lambda_{i_{0}}$.
\end{proposition}

\begin{proof}
  We have for all $\omega \in \R, \eta > 0$
  \begin{align*}
    \left\|(V_{\rm tot} + K_{0}) \frac 1 {-\tfrac 1 2 \Delta - \Lambda +\ii(\omega+\ii\eta) J}\right\|_{L^{2}_{\alpha} \to L^{2}_{\alpha}} &\le \|V_{\rm tot}+K_{0}\|_{H^{2}_{-\alpha}\to L^{2}_{\alpha}} \|(-\tfrac 1 2 \Delta - \Lambda\pm(\omega+i\eta))^{-1}\|_{L^{2}_{\alpha}\to H^{2}_{-\alpha}}\\
    &\le\|V_{\rm tot}+K_{0}\|_{H^{2}_{-\alpha}\to L^{2}_{\alpha}} \|(-\tfrac 1 2 \Delta - \Lambda\pm(\omega+i\eta))^{-1}\|_{L^{2} \to H^{2}}\\
    &\le\frac 1 \eta \|V_{\rm tot}+K_{0}\|_{H^{2}_{-\alpha}\to L^{2}_{\alpha}}
  \end{align*}
  so that $1+ (V_{\rm tot} + K_{0}) (-\tfrac 1 2 \Delta - \Lambda +\ii(\omega+\ii\eta) J)^{-1}$ is
invertible for $\eta$ large enough in $L^{2}_{\alpha}$. It follows from
the analytic Fredholm theorem (Theorem C.8 of
\cite{dyatlov_mathematical_2019}) that
$1+ (V_{\rm tot} + K_{0}) (-\tfrac 1 2 \Delta - \Lambda +\ii(\omega+\ii\eta) J)^{-1}$ has a
meromorphic continuation to a complex neighborhood of
$\lambda_{a_{0}}-\lambda_{i_{0}}$, and the result follows by \eqref{eq:M_as_pert_of_Mess}.
\end{proof}
This proves the first part of Theorem \ref{th:resonances}.

\subsection{Resonances}
We now  prove the second part of Theorem~\ref{th:resonances}.
We now investigate the poles of $(M_{\mathrm{dyn}} +\ii z J)^{-1}$ on ${\rm Ran}(1-P_0)$, which are also the poles of $(M+\ii z J)^{-1}$, in the
asymptotic regime where $K_0$ is small.

For $K_{0}=0$, in the Casida representation,
\begin{align}
  \Omega +\ii z J \simeqcasida
  \begin{pmatrix}
    \Omega - z& 0\\
    0& \Omega + z
  \end{pmatrix}
\end{align}
Near $z = \lambda_{a_{0}}-\lambda_{i_{0}} > 0$, the block $\Omega + z$ is always invertible for
${\rm Im}(z) > 0$. The operator $\Omega +\ii(\lambda_{a_{0}}-\lambda_{i_{0}}) J$ has a
simple zero eigenvalue with eigenvector
\begin{align}
  \label{eq:9}
  (U_{i_{0}\to a_{0}})_{i} \simeqcasida
  \delta_{i,i_{0}}
  \begin{pmatrix}
    \psi_{a_{0}} \\0
  \end{pmatrix}
\end{align}
where $H_{0} \psi_{a_{0}} = \lambda_{a_0} \psi_{a_{0}}$ with
$\|\psi_{a_{0}}\|=1$. It also has
continuous spectrum at $0$ in all the ionized sectors $i$ such that
$\lambda_{a_{0}} - \lambda_{i_{0}} > -\lambda_{i}$. By the results of
the previous section, the inverse of its analytic continuation from
$L^{2}_{\alpha}$ to $H^{2}_{-\alpha}$ has a single pole at
$\lambda_{a_{0}}-\lambda_{i_{0}}$, with residue
$P_{i_{0}\to a_{0}} = |U_{i_{0}\to a_{0}}\rangle\langle  U_{i_{0}\to a_{0}}|$.



We now split the space $(L^{2})^{2N}$ orthogonally in ${\rm Ran}(P_{i_{0}\to a_{0}})$ and
${\rm Ran}(1  - P_{i_{0}\to a_{0}})$. By a perturbation argument, there exists a complex neighborhood of
$\lambda_{a_{0}}-\lambda_{i_{0}}$ inside which, for
$\delta = \|K_0\|_{L^2_\alpha \to H^2_{-\alpha}}$ small enough, the
orthogonal restriction on ${\rm Ran}(1 - P_{i_{0}\to a_{0}})$ of the operator
$M+\ii zJ$ is invertible; let
\begin{align*}
  R^{\perp}(z) = \left((M+\ii zJ)|_{{\rm Ran}(1 - P_{i_{0}\to a_{0}})}\right)^{-1}
\end{align*}
this inverse. By a Schur
complement, the operator $M +\ii z J$ is not invertible as an operator from
$L^{2}_{-\alpha}$ to $L^{2}_{\alpha}$ if and only if the scalar Schur
complement $S(z)$ vanishes, where
\begin{align}
  \label{eq:10}
  S(z) &= \langle  U_{i_{0}\to a_{0}}|M +\ii z J|U_{i_{0}\to a_{0}} \rangle - \langle  U_{i_{0}\to a_{0}}|(M+\ii zJ) R^{\perp}(z)(M+\ii zJ)|U_{i_{0}\to a_{0}} \rangle\\
  &= (\lambda_{a_{0}}-\lambda_{i_{0}} - z) + \langle  U_{i_{0}\to a_{0}}|K_{0}|U_{i_{0}\to a_{0}} \rangle - \langle  U_{i_{0}\to a_{0}}|(M+\ii zJ) R^{\perp}(z)(M+\ii zJ)|U_{i_{0}\to a_{0}} \rangle.
\end{align}
This is an analytic equation in $z$, and the third term is of order
$2$ in $\delta$ for $z$ in a small enough neighborhood of $\lambda_{a_0}-\lambda_{i_0}$. It follows by the
implicit function theorem that, for
$\delta$ small enough, there is a
zero $z_{\rm pole}$ of $S$ close to $\lambda_{a_{0}}-\lambda_{i_{0}}$
such that
\begin{align}
  \label{eq:13}
  z_{\rm pole} = \lambda_{a_{0}}-\lambda_{i_{0}} + \langle  U_{i_{0}\to a_{0}}|K_0|U_{i_{0}\to a_{0}} \rangle - \lim_{\eta \to 0^{+}}\langle  U_{i_{0}\to a_{0}}|K_0 R^{\perp}(\lambda_{a_{0}}-\lambda_{i_{0}}+\ii \eta) K_0|U_{i_{0}\to a_{0}} \rangle + O(\delta^{3}).
\end{align}

By the Sokhotski-Plemelj formula, the skew-adjoint part of
$\lim_{\eta\to 0^{+}}R^{\perp}(\lambda_{a_{0}}-\lambda_{i_{0}}+\ii \eta)$ is given by
\begin{align}
  \label{eq:14}
  \lim_{\eta \to 0^{+}}{\mathcal A}(R^{\perp}(\lambda_{a_{0}}-\lambda_{i_{0}}+\ii \eta)) \simeqcasida
  \ii\pi \sum_{\substack{1 \leq i \leq N\\ \lambda_{a_{0}}-\lambda_{a_{0}} + \lambda_{i}>0}}(\Pi_{i}^{+})^{*}p_{H_{0}}(\lambda_{a_{0}}-\lambda_{i_{0}} + \lambda_{i}) \Pi_{i}^{+}
\end{align}
where $p_{H_{0}}(\lambda)d\lambda$ is the projection-valued measure
associated with $H_{0}$ and
\begin{align*}
  (\Pi_{i}^{+} U)_{j} \simeqcasida
  \delta_{ij}
  \begin{pmatrix}
     u_{j}\\0
  \end{pmatrix}
\end{align*}
is the projection on the upper block, $i$-th sector of $(L^{2})^{2N}$.

It follows that
\begin{align}
  {\rm Im}(z_{\rm pole}) = - \pi \sum_{i, \lambda_{a_{0}}-\lambda_{a_{0}} + \lambda_{i}>0} \langle  U_{i_0 \to a_0} |K_{0}(\Pi_{i}^{+})^{*}p_{H_{0}}(\lambda_{a_{0}}-\lambda_{i_{0}} + \lambda_{i}) \Pi_{i}^{+}K_{0} |U_{i_0 \to a_0}  \rangle+ O(\delta^{3}).
  \label{eq:15}
\end{align}
It remains to check that the residue at this pole is nonzero. 
By definition of $\chi$, it is the case if $\langle S_0^*V_P, U_{i_0 \to a_0} \rangle \not=0$ for some $V_P \in H^2$, \textit{i.e.} $S_0U_{i_0 \to a_0} = \psi_{i_0}\overline{\psi}_{a_0} + \psi_{a_0}\overline{\psi}_{i_0} \not=0$.

\section*{Acknowledgement}

This project has received funding from the European Research Council (ERC) under the European Union’s Horizon 2020 research and innovation program (grant agreement No. 810367).

\appendix
\addcontentsline{toc}{section}{Appendices}
\section*{Appendices}

\section{Notation memo }\label{sec:appendix_tables}

The main notations related to the complex structure of the space are reminded in table \ref{tab:notations}. 
Table \ref{table_representations} summarizes the representation of conjugation operations within the Casida and Real/Imaginary representations.  
The notations of the most important operators of the paper are gathered in table \ref{tab:notations2}.
\begin{table}[h!] 
\begin{center}
\begin{tabular}{|c|c|}
\hline
\textbf{Symbol} & \textbf{Explanation} \\\hline & \\  [-1em] 
$\mathcal{Y} = (L^2(\mathbb{R}^3, \mathbb{C}))^N$ & \\ \hline
$\jj$ & Imaginary unit in $\mathcal{Y}$ \\ \hline 
$Y_{\mathbb{R}} = (\mathcal{Y}, \mathbb{R})$ & $\mathcal{Y}$ endowed with a real vector space structure  \\ \hline
$J$ & Representation of the multiplication operator by $\jj$ in $Y_{\mathbb{R}}$ \\  \hline
$\ii$ & Imaginary unit in the complexification of $Y_{\mathbb{R}}$  \\ \hline
$Y= Y_{\mathbb{R}}+ \ii   Y_{\mathbb{R}}=(L^2)^{2N}$ & Complexification of $Y_{\mathbb{R}}$ \\\hline
\end{tabular}
\end{center}
\vspace{-10pt}
\caption{\label{tab:notations}Collection of the main notations related to the complexification of space.}
\end{table}

\begin{table}[h!]
  \centering
  \begin{tabular}[h!]{|c|c|c|}
    \hline
      &\textbf{Casida representation }&\textbf{Real/imaginary representation}\\
    \hline
    $U_{\rm r} + \jj U_{\rm j}$ for $U_{\rm r},U_{\rm j} \in (L^2(\mathbb{R}^3, \mathbb{R}))^N$ &$
                                  \begin{pmatrix}
                                    U_{\rm r} + \ii U_{\rm j} \\ U_{\rm r} - \ii U_{\rm j}
                                  \end{pmatrix} \in
                                         \left\{\begin{pmatrix}
                                             U\\\overline U
                                           \end{pmatrix}, U \in (L^2(\mathbb{R}^3, \mathbb{C}))^N\right\}$&$
                                                  \begin{pmatrix}
                                                    U_{\rm r}\\U_{\rm j}
                                                  \end{pmatrix} \in (L^2(\mathbb{R}^3, \mathbb{R}))^{2N}$ \\
    \hline
    $U+\ii V$ for $U,V \in (L^2(\mathbb{R}^3, \mathbb{C}))^N$& 
                              $\begin{pmatrix}
                                U+\ii V\\  \overline{U} + \ii  \overline{V}
                              \end{pmatrix} \in (L^2(\mathbb{R}^3, \mathbb{C}))^{2N}
                              $  & $
                                  \begin{pmatrix}
                                    U_{\rm r}+\ii V_{\rm r}\\
                                    U_{\rm j} + \ii V_{\rm j}
                                  \end{pmatrix} \in (L^2(\mathbb{R}^3, \mathbb{C}))^{2N}
                                  $\\
    \hline
    $J$ (multiplication by $\jj$) & $
                            \begin{pmatrix}
                              \ii &0\\0&-\ii
                            \end{pmatrix}
                            $&$
                               \begin{pmatrix}
                                 0&-1\\
                                 1&0
                               \end{pmatrix}$
    \\
    \hline
    $U \mapsto A U + B \overline U$ & $\begin{pmatrix} A & B\\\overline{B}&\overline{A}  \end{pmatrix}$ &
    $\begin{pmatrix} A_{\rm r}+B_{\rm r} & -A_{\rm j}+B_{\rm j}\\A_{\rm j}+B_{\rm j}&A_{\rm r}-B_{\rm r}  \end{pmatrix}$\\
    \hline
  \end{tabular}
  \caption{Summary of the different representations of the operators in $Y$.}\label{table_representations}
\end{table}

\begin{table}[h!] 
  \begin{center}
   \begin{tabular}{|p{9cm}|p{6.5cm}|p{1cm}|}\hline
    \parbox{9cm}{\centering \textbf{Definition}}  &  \parbox{6.5cm}{\centering \textbf{Remarks}} & \parbox{1cm}{\centering \textbf{Eq.}}\\ \hline
    $~~~~H_0=H[\rho_0]$ &   Hamiltonian at $\Psi^0$, acting on $L^2$ or $(L^2)^N$ &  \parbox{1cm}{\centering\eqref{eq:Hamiltonian}}\\ 
    \hline
    $~~~~P_0: \begin{cases} L^2(\mathbb{R}^3, \mathbb{C})^N  & \rightarrow  L^2(\mathbb{R}^3, \mathbb{C})^N  \\ 
    U & \mapsto  \left(\sum_{j=1}^N|\psi_j^0 \rangle \langle \psi_j^0| u_i\right)_{i=1}^N \end{cases}  $ & Projector on the occupied orbitals & \parbox{1cm}{\centering\eqref{eq:16}}\\
    \hline
    $~~~~\mathcal{S}_0: \begin{cases} L^2(\mathbb{R}^3, \mathbb{C})^N &\rightarrow \mathbb{C}  \\
    U &\mapsto \sum_{j=1}^N \overline{\psi_j^0} u_j +  \psi_j^0 \overline{u_j} \end{cases} $ & Orbital variation to density variation map & \parbox{1cm}{\centering\eqref{eq:Spsi}}  \\
    \hline
    $~~~\mathcal{K}_0: \begin{cases} L^2(\mathbb{R}^3, \mathbb{C})^N &\rightarrow  L^2(\mathbb{R}^3, \mathbb{C})^N  \\ U &\mapsto \left(\mathcal{S}_0(U)*\frac{1}{|r|} + v_{\rm xc}'(\rho_0 )\mathcal{S}_0(U)\right) \Psi^0\end{cases} $ & Unconstrained Hessian of the energy &\parbox{1cm}{\centering\eqref{eq:def_K}} \\
    \hline
    $~~~~~\Lambda: \begin{cases} L^2(\mathbb{R}^3, \mathbb{C})^N & \rightarrow L^2(\mathbb{R}^3, \mathbb{C})^N \\ U & \mapsto \left( \lambda_i u_i\right)_{i=1}^N \end{cases} $ & Multiplication by the eigenvalues of $H_0$ & \parbox{1cm}{\centering\eqref{eq:def_Lambda}}\\
    \hline
    $~~~~~\Omega: \begin{cases} L^2(\mathbb{R}^3, \mathbb{C})^N & \rightarrow L^2(\mathbb{R}^3, \mathbb{C})^N \\ U & \mapsto (H_0-\Lambda) U \end{cases}$ & Non-interacting constrained Hessian & \parbox{1cm}{\centering\eqref{eq:def_Omega}}\\
    \hline
    $\mathcal{M}_{\rm dyn}:\begin{cases}L^2(\mathbb{R}^3, \mathbb{C})^N & \rightarrow L^2(\mathbb{R}^3, \mathbb{C})^N \\U& \mapsto \Omega + \mathcal{K}_0\end{cases} $ & Constrained Hessian &\parbox{1cm}{\centering\eqref{eq:def_M_dyn}}\\  & & \\[-1em]
    \hline 
    $~~~~~S_0,~K_0,~ M_{\rm dyn}$ & $\mathbb{C}$-linear extensions of $\mathcal{S}_0, ~\mathcal{K}_0, \mathcal{M}_{\rm dyn}$ to $Y$ &  \\
    \hline
    $~~~~~M= (1-P_0)M_{\rm dyn}(1-P_0)$& $M_{\rm dyn}$ restricted to $\Big\{\rm Ran (\psi_i^0)_{i=1}^N\Big\}^\perp$ & \parbox{1cm}{\centering\eqref{eq:21}}\\
    \hline
  \end{tabular}
  \end{center}
  \vspace{-10pt}
  \caption{\label{tab:notations2}Notations related to the main operators of the paper.}
  \end{table}

\section{Second-order expansion of the energy}
\label{sec:energy}
\begin{proof}[Proof of Proposition \ref{generalites_E_H_K}]
\label{sec:proof_generalites_E_H_K}
The energy is
\begin{align}
  \mathcal{E}(\Psi)&=\sum_{i=1}^N\left( \frac{1}{2}\int_{\mathbb{R}^3} |\nabla \psi_i(r)|^2 + \int_{\mathbb{R}^3} V_{\rm ext}(r)|\psi_i(r)|^2 dr \right)+ \frac{1}{2} \int_{\mathbb{R}^3} \frac{\rho_{\Psi}(r) \rho_{\Psi}(r')}{|r-r'|}drdr' + \int_{\mathbb{R}^3} e_{\rm xc}(\rho_{\Psi}(r)) dr
\end{align}
We treat these terms in order. The first term is clearly smooth from
$H^{2}$ to $\R$. Since $H^{2}$ is an algebra, the mapping
$\psi_{i} \mapsto |\psi_{i}|^{2}$ is smooth from $H^{2}$ to $H^{2}$,
and so by
\begin{align}
  \left| \int_{\R^{3}} (V_{1} + V_{2}) |\psi_{i}|^{2} \right| \le \|V_{1}\|_{L^{2}}\||\psi_{i}|^{2}\|_{L^{2}} + \|V_{2}\|_{L^{\infty}} \|\psi_{i}\|_{L^{2}}^{2},
\end{align}
the second term is also smooth from $H^{2}$ to $\R$. From Lemma
\ref{lemma_stabilite}, the map from $\Psi$ to the operator
$L_{\Psi}$ of multiplication by $\rho_{\Psi} \ast \frac 1 {|r|}$ is smooth
from $H^2$ to $\mathcal L(H^2, H^2)$, and therefore the third term
\begin{align}
  \frac{1}{2} \int_{\mathbb{R}^3} \frac{\rho_{\Psi}(r) \rho_{\Psi}(r')}{|r-r'|}drdr' = \frac 1 2 \sum_{i=1}^{N}\langle \psi_{i}| L_{\Psi} \psi_{i} \rangle
\end{align}
is smooth from $H^{2}$ to $\R$.

For all $\Psi \in (H^{2})^{N}$, $\Psi$ is bounded so $\rho$ is
also bounded. Since $e_{\rm xc}(0)=0$, there is $C_{\Psi}$ such that
$|e_{\rm xc}(\rho)| \le C_{\Psi} \rho$. Since $\rho$ is integrable, so
is $e_{\rm xc}(\rho)$. Furthermore, since $ e_{\rm xc}$ is $\mathcal{C}^2$:
\begin{align}
  e_{\rm xc}(\rho)&= e_{xc}(0)+\rho e_{\rm xc}'(0)+ \frac{\rho^2}{2}e_{xc}''(0)+ o(\rho^2)
\end{align}
Since $e_{\rm xc}(0)=0$ and $\rho$ as well as $\rho^2$ are integrable for $\rho \in H^2$, $\Psi \mapsto \int_{\mathbb{R}^3 } e_{xc}(\rho_{\Psi}(r))dr$ is $\mathcal{C}^2$ from $H^2$ to $\mathbb{R}$. 
\end{proof}

\begin{proof}[Proof of Proposition
  \ref{proposition:operator_M_dyn}]
  Let $U^{1} \in (H^{2})^{2N} \cap {\rm Ran}(1-P_{0})$. For all
  $\varepsilon > 0$, let
  \begin{align*}
    \Psi(\varepsilon) = {\rm Ortho}(\Psi^{0} + \varepsilon U^{1}),
  \end{align*}
  where
  \begin{align*} {\rm Ortho}(\Psi) = \Psi (\Psi^{*}\Psi)^{-1/2}.
  \end{align*}
  We have the expansion
  \begin{align*}
    \Psi(\varepsilon) = \Psi^{0} + \varepsilon U^{1} + \varepsilon^{2} U^{2} + O(\varepsilon^{3})
  \end{align*}
  in $H^{2}$.

  Since $\Psi(\varepsilon)^{*}\Psi(\varepsilon)=1$ for all
  $\varepsilon$, by identification, we get
  \begin{align}
    (U^{2})^{*} U^{0} + (U^{0})^{*} U^{2} = -  (U^{1})^{*} U^{1}. \label{eq:critere_2_manifold}
  \end{align}
  We can then compute
  \begin{align*}
    \mathcal{E}(\Psi(\varepsilon))&= \mathcal{E}(\Psi^0)+ 2{\rm Re}\langle H_0\Psi^0 | \varepsilon  U^1 +\varepsilon^2  U^2 \rangle + \langle \varepsilon  U^1|H_0|\varepsilon  U^1\rangle + {\rm Re}\langle \varepsilon  U^1| \mathcal{K}_{0}(\varepsilon  U^1)\rangle  +o(\varepsilon^2)\\
    &=  \mathcal{E}(\Psi^0)+ \varepsilon^2 \Big(2{\rm Re} \langle H_0  \Psi^0 |  U^2\rangle + \langle  U^1|H_0| U^1\rangle+ {\rm Re}\langle  U^1 | \mathcal{K}_{0} ( U^1) \rangle\Big) + o (\varepsilon^2)\\
    2 {\rm Re}\langle H_0  \Psi^0 |  U^2\rangle&= 2{\rm Re}\left(\sum_{i=1}^N \lambda_i \langle  \psi^0_i|u_i^2\rangle\right)\\
    &= -\sum_{i=1}^N \lambda_i \langle  u_i^1|  u_i^1\rangle 
    \mbox{ by \eqref{eq:critere_2_manifold}}
  \end{align*}
  so that
  \begin{align}
    \mathcal{E}(\Psi(\varepsilon))&= \mathcal{E}(\Psi^0) + \varepsilon^2 \langle  U^1|\mathcal{M}_{\rm dyn}( U^1)\rangle+o(\varepsilon^2).
  \end{align}

  On the other hand,
  we can solve the orthogonal Procustes problem
  \begin{align*}
    \min_{R \in {\rm U}(N)} \|\Psi - \Psi^{0} R\|^{2} &= 2 N - 2\max_{R \in{\rm U}(N)} {\rm Re} \langle \Psi, \Psi^{0} R \rangle\\
    &= 2 N - 2\max_{R \in{\rm U}(N)} {\rm Re} \langle  (\Psi^{0})^{*} \Psi ,R\rangle\\
    &=2 N - 2\max_{R \in{\rm U}(N)} {\rm Re} \langle  (\Psi^{0})^{*} \Psi ,R\rangle\\
    &=2N - 2 \max_{R' \in{\rm U}(N)} {\rm Re} \langle  \Sigma  ,R'\rangle\\
    &= 2N - 2 {\rm Tr}(\Sigma)\\
  \end{align*}
  where $\Sigma$ is the diagonal matrix of the singular values of
  $(\Psi^{0})^{*} \Psi$. We have
  \begin{align*}
    (\Psi^{0})^{*} \Psi &= 1 + \varepsilon^{2} (U^{0})^{*} U^{2} + O(\varepsilon^{3})\\
    ((\Psi^{0})^{*} \Psi)^{*} ((\Psi^{0})^{*} \Psi) &= 1 + \varepsilon^{2} \left((U^{0})^{*} U^{2} + (U^{0})^{*} U^{2}\right) + O(\varepsilon^{3})\\
    &= 1 + 2\varepsilon^{2} (U^{1})^{*}U^{1} + O(\varepsilon^{3})\\
    {\rm Tr}(\Sigma)&= {\rm Tr}\left( \sqrt{((\Psi^{0})^{*} \Psi)^{*} ((\Psi^{0})^{*} \Psi)} \right)\\
    &= N + \varepsilon^{2} {\rm Tr} \left( (U^{1})^{*} U^{1} \right) + O(\varepsilon^{3})
  \end{align*}
  so that
  \begin{align*}
    \min_{R \in {\rm U}(N)} \|\Psi - \Psi^{0} R\|^{2} &=2\varepsilon^{2} \|U^{1}\|^{2} + O(\varepsilon^{3})
  \end{align*}
  and the result follows.
  
\end{proof}

\section{Control of the Coulomb terms}

  We state a useful technical lemma, which ensures stability in $H^2$ in several occasions through the article. 

  \begin{lemma}\label{lemma_stabilite}
      Let $f$, $g$, $h$ three functions of $H^2$. Then the function $(fg*\tfrac{1}{|r|})h$ is in $H^2$, and there exists a constant $c$ such that:
      \begin{align}
          \|(fg*\tfrac{1}{|r|})h\|_{H^2} \leq c \|f\|_{H^2} \|g\|_{H^2} \|h\|_{H^2}.
      \end{align} 
  \end{lemma}
  
  The proof relies on the Hardy-Littlewood-Sobolev (HLS) inequality in dimension 3:
  \begin{align}
      \forall p, q >0;~ 1 < p,q<\infty, ~s.t.~ 1 +\frac{1}{p} = \frac{1}{q} + \frac{1}{3},\\
      \exists C_{HLS}^{p,q}>0, ~ \forall a \in L^q: ~
      a*\tfrac{1}{|r|} \in L^p \mbox{ and } \|a* \frac{1}{|r|}\|_{L^p} \leq C_{HLS}^{p, q} \|a\|_{L^q} .
  \end{align}
  
  \begin{lemma}\label{lemma_ab}
      Let $a$ and $b$ two complex valued functions such that $a$ is in $L^{1}\cap L^{\tfrac{3}{2}}$ and $b$ is in $L^{t}$, with $t$ in $(2, 6)$. 
       Then $(a*\tfrac{1}{|r|})b$ is $L^2$, and $\|(a*\tfrac{1}{|r|})b\|_{L^2} \leq c  (\|a\|_{L^1}+ \|a\|_{L^{\frac{3}{2}}})\|b\|_{L^{t}}$, where $c$ is an unimportant constant. 
  \end{lemma}
  \begin{proof}[Proof of Lemma \ref{lemma_ab}]~
  
      Let $p$ such that $\frac{1}{p}+\frac{1}{t}= \frac{1}{2}$, with $t$ in $(2, 6)$. Then $p$ is in $(3, +\infty)$. Thus there exists $q$ in $(1, \frac{3}{2})$ such that $p$ and $q$ are an admissible pair for the HLS inequality. Since $a$ is in $L^1 \cap L^{\tfrac{3}{2}}$, it is in $L^{q}$ as well and we can write:
      \begin{gather}
          \|a* \frac{1}{|r|}\|_{L^{p}} \leq C_{HLS}^{p, q} \|a\|_{L^{q}}
      \end{gather}
      By the Hölder inequality, $(a* \frac{1}{|r|})b$ is $L^2$ and:
      \begin{align}
          \|\left(a* \frac{1}{|r|}\right)b\|_{L^{2}}&\leq C_{HLS}^{p, q} \|a\|_{L^{q}} \|b\|_{L^{t}}\\
          \|\left(a* \frac{1}{|r|}\right)b\|_{L^{2}}&\leq C_{HLS}^{p, q} (\|a\|_{L^{1}}+\|a\|_{L^{\frac{3}{2}}})\|b\|_{L^{t}}
      \end{align}
  \end{proof}
  \begin{figure}[!h]
      \centering
           \includegraphics[width=0.8\textwidth]{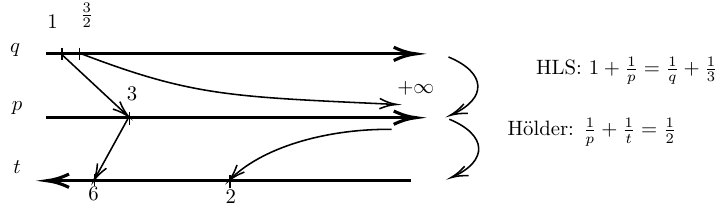}
           \caption{Choices of $p$, $q$, $t$ to apply Lemma \ref{lemma_ab}}
   \end{figure}

  \begin{proof}[Proof of Lemma \ref{lemma_stabilite}]~
  
  We prove that $\left(fg*\tfrac{1}{|r|}\right)h$ is $L^2$ and then that its second derivative is $L^2$ as well. The proof is the application of Lemma \ref{lemma_ab}. 
  
  $fg$ is in $H^2$, it thus belongs to $L^\infty \cap L^2$. It is $L^1$ as well by the Hölder inequality, since both $f$ and $g$ are $L^2$. Therefore $fg$ is $L^1 \cap L^{\tfrac{3}{2}}$. $h$ is $L^2 \cap L^\infty$, in particular it is $L^{t}$ for any $t$ in $(2, 6)$, and we can apply Lemma \ref{lemma_ab}.
  
  We now show that the second derivative of the function belongs to $L^2$ . It only contains terms of the following form:
  \begin{itemize}
      \item $ \left(f \nabla^2 g * \tfrac{1}{|r|}\right) h$. $f$ and $\nabla^2g$ are both $L^2$, thus their product is $L^1$. Since $f$ is bounded and $\nabla^2g$ is $L^2$, $f\nabla^2g$ is $L^{2}$ as well and Lemma \ref{lemma_ab} can be applied. 
      \item $ \left(\nabla f \nabla g * \tfrac{1}{|r|}\right)  h$. $\nabla f$ and $\nabla g$ are both $L^2$, thus their product is $L^1$. They are also both in $H^1$, which injects itself continuously (by Sobolev injection) in $L^6$ in dimension 3. Their product is thus in $L^3$.
      \item $ \left(f \nabla g* \tfrac{1}{|r|}\right)  \nabla h$. $f\nabla g$ is $L^1 \cap L^2$. $\nabla h$ belongs to $H^1$, which injects itself continuously in $L^6$ by the Sobolev inequality. It also belongs to $L^2$.
      \item $ \left(fg * \tfrac{1}{|r|} \right) \nabla^2 h$. In this case, the HLS inequality cannot be directly applied, since $h$ is only $L^2$. We would need $\left(fg*\tfrac{1}{|r|}\right)$ to be $L^\infty$ for the Hölder inequality, which is not an admissible $p$ exponent. However, boundedness can easily be obtained like this:
      \begin{align}
          \forall r \in \mathbb{R}^3, |\left(fg*\tfrac{1}{|r|}\right)(r)|& \leq \int_{\mathbb{R}^3} \frac{|(fg)(r')|}{|r-r'|} dr'\\
          &= \int_{\mathcal{B}_{r, 1}} \frac{|(fg)(r')|}{|r-r'|} dr' + \int_{\mathbb{R}^3 \backslash \mathcal{B}_{r, 1}}\frac{|(fg)(r')|}{|r-r'|} dr' \\
          &\leq   \int_{\mathcal{B}_{0, 1}} \frac{1}{|r'|}dr' \sup_{ r' \in \mathbb{R}^3} |(fg)(r')|+ \int_{\mathbb{R}^3}|(fg)(r')|dr'\\
          &\leq  K\|f\|_{L^\infty} \|g\|_{L^\infty} + \|f\|_{L^2} \|g\|_{L^2} \quad \mbox{with} \quad K= \int_{\mathcal{B}_{0, 1}} \frac{1}{|r'|}dr' .
      \end{align}
  \end{itemize}
  
  \end{proof}

\section{Technical lemmas on \texorpdfstring{$S_0$}{S0} and \texorpdfstring{$K_0$}{K0} }

\begin{lemma}
\label{lem:S_et_K}
The following assertions are true:
\begin{enumerate}
  \item $S_0 : (H^2)^{2N} \to (H^2)^{2}$ is bounded;
  \item $K_0 : (H^1)^{2N}\to (L^2)^{2N}$ is bounded;
  \item $K_0 : (H^2)^{2N}\to (H^2)^{2N}$ is bounded;
  \item $K_0 : (H^2)^{2N} \to (L^2)^{2N}$ is compact;
  \item For $\alpha>0$ small enough, $K_0 : (H_\alpha^2)^{2N} \to (L_{-\alpha}^2)^{2N}$ is bounded.
\end{enumerate}
\end{lemma}

\begin{proof}
  \begin{enumerate}
    \item \textbf{$S_0 : (H^2)^{2N} \to (H^2)^{2}$ is bounded}

    This follows from Assumption~\ref{assumption3} and the fact that $H^2$ is an algebra.

    \item \textbf{$K_0 : (H^1)^{2N}\to (L^2)^{2N}$ is bounded}

    Again by Assumption~\ref{assumption3}, for $U \in (H^1)^{2N}$, $S_0(U) \in (H^1)^{2N}$, so by assumption on $v_\mathrm{xc}$, $v'_\mathrm{xc}(\rho_0) S_0(U) \Psi_0 \in (H^1)^{2N}$.
    For $S_0(U) * \frac{1}{|r|}$, we use a Hardy inequality to deduce that $\|S_0(U) * \frac{1}{|r|}\|_{L^2} \leq \|U\|_{H^1}$.
    This proves the assertion. 

    \item \textbf{$K_0 : (H^2)^{2N}\to (H^2)^{2N}$ is bounded}
    
    By assumption on $v_\mathrm{xc}$, and since $H^2$ is an algebra, for $U \in (H^2)^{2N}$, $v'_\mathrm{xc}(\rho_0) S_0(U) \Psi_0 \in (H^2)^{2N}$. 
    By Lemma~\ref{lemma_stabilite}, we have that $S_0(U)*\frac{1}{|r|} \Psi_0 \in (H^2)^{2N}$ for $U \in (H^2)^{2N}$. 
    This finishes the proof of this item.
  
    \item \textbf{$K_0 : (H^2)^{2N} \to (L^2)^{2N}$ is compact}
    
    For the proof of the compactness of $K_0$, we have for $U \in (H^2)^{2N}$
    \[
        K_0(U) = \left( S_0(U) * \frac{1}{|\cdot|} + v'_{xc}(\rho) S_0(U) \right) \Psi_0.
    \]
    $V \mapsto v'_{xc}(\rho) S_0((-\Delta+1)^{-1})V) \Psi_0$ is an operator of the form $f(x) g(-\ii \nabla)$ with $f,g \in L^2$, hence it is Hilbert-Schmidt~\cite[Theorem 4.1]{barry} thus compact.
    By Lemma~\ref{lemma_stabilite}, for $V \in (L^2)^{2N}$, $S_0((-\Delta + 1)^{-1}V) * \frac{1}{|\cdot|}\Psi_0$ is in $(H^2)^{2N}$. 
    Using that $\Psi_0 \in H^2_\alpha$ we have exponential decay of $\Psi_0$ and its derivative, so by the Rellich-Kondrachov compactness embedding, $V \mapsto S_0((-\Delta + 1)^{-1}V) * \frac{1}{|\cdot|}\Psi_0$ is compact on $(L^2)^{2N}$.

    \item \textbf{$K_0 : (H_\alpha^2)^{2N} \to (L_{-\alpha}^2)^{2N}$ is bounded}
    
    This follows from the assumption on the exponential decay of the eigenfunction.
  \end{enumerate}
    
\end{proof}

\bibliography{biblio}
\bibliographystyle{unsrt}

\end{document}